\documentclass[a4paper,12pt]{amsart}
\newcommand{\query}[1]%
{\mbox{}\marginpar{\raggedright\hspace{0pt}{\small\em #1}}}%

\usepackage{amssymb}
\usepackage{amsmath}
\usepackage{amsfonts}
\usepackage{mathrsfs}
\usepackage[all]{xy}
\usepackage{MnSymbol} 
\usepackage{tikz}

\newcommand{\lin}{\vlin{8.5}}
\newcommand{\cou}[2]{\unitlength1mm\begin{picture}(0,8)
    \put(0,0){\circle{1.5}}
    \put(0,3){\makebox(0,5)[b]{$#1$}}
    \put(0,-7){\makebox(0,4)[t]{$#2$}}
      \end{picture}
      \rule[-7mm]{0mm}{7mm}}
\newcommand{\vlin}[1]{\hspace{0.75mm}\unitlength1mm\begin{picture}(#1,0)
                       \put(0,0){\line(1,0){#1}}
                      \end{picture}\hspace{0.75mm}\rule[-3mm]{0mm}{4mm}}

\usepackage{hyperref}
\usepackage{enumitem}

\usepackage[draft,multiuser,english]{fixme}

\FXRegisterAuthor{p}{anp}{AP}

\FXRegisterAuthor{l}{kl}{KL}

\theoremstyle{plain}
\newtheorem{theorem}{Theorem}[section] 
\newtheorem{lemme}[theorem]{Lemma}
\newtheorem{lemma}[theorem]{Lemma}
\newtheorem{proposition}[theorem]{Proposition}
\newtheorem{corollaire}[theorem]{Corollary}
\newtheorem{corollary}[theorem]{Corollary}
\newtheorem*{theorem*}{Theorem}
 
\renewcommand\varrho\eta

\theoremstyle{definition}
\newtheorem{definition}[theorem]{Definition}
\newtheorem{rappel}[theorem]{}

\newtheorem{question}[theorem]{Question}
\newtheorem*{question*}{Question}
\newtheorem{example}[theorem]{Example}

\newtheorem{convention}[theorem]{Convention}

\theoremstyle{remark}
\newtheorem{remarque}[theorem]{Remark}
\newtheorem{remark}[theorem]{Remark}

\newcommand{\K}{{\mathbf{k}}}
\newcommand{\G}{{\mathscr{G}}}

\renewcommand{\div}{\operatorname{div}}

\def\bangle#1{\langle #1 \rangle}
\def\relint{\operatorname{rel.int.}}

\def\LND{\operatorname{LND}}

\def\Hom{\operatorname{Hom}}
\def\Spec{\operatorname{Spec}}

\def\Cone{\operatorname{Cone}} 
\def\Conv{\operatorname{Conv}}

\def\ord{\operatorname{ord}}
\def\aff{\operatorname{aff}}
\def\dim{\operatorname{dim}}
\def\codim{\operatorname{codim}}

\def\coloneq{\mathrel{\mathop:}=}
\renewcommand\colon{\mathrel{\mathop:}}

\def\Rt{{\mathrm{Rt}}}
\def\Rs{{\mathrm{Rt}_{\rm ss}}}
\renewcommand\deg{{\mathrm{deg}\,}}
\def\SL{{\mathrm{SL}}}

\def\I{{\mathrm{I}}}
\def\TT{{\mathbb{T}}}
\def\PP{{\mathbb P}}
\def\ZZ{{\mathbb Z}}

\def\V{{\mathbb V}}
\def\QQ{{\mathbb Q}}
\def\KK{{\mathbb K}}
\def\G{{\mathbb G}}

\def\AA{{\mathbb A}}
\def\D{{\mathfrak D}}
\def\F{{\mathcal F}}
\def\VV{{\mathbb V}}
\def\Vc{{\mathcal V}}
\def\cP{{\mathcal P}}
\def\Cc{{\mathcal C}}
\def\Oc{{\mathcal O}}
\def\K{{\mathbb K}}

\def\Emb{{\rm Emb}(G/H)}
\renewcommand\phi\varphi
\def\vrho{\varrho}
\def\Ac{A(\cP,\F)}
\def\sistar{\sigma(1)^{\star}}
\def\dcol{\mathfrak{D}_{\bullet}}
\def\lins{{\rm lin.}(\sigma^{\vee})}
\def\linea{{\rm lin.}}
\def\VV{\mathbb{V}}
\def\ee{{\theta}}
\def\LG{\mathrm{LND}_{G}}
\def\Oo{\mathscr{O}}
\def\Sch{{\rm Sch}(G/H)}
\def\vvert{\Vc_{\rm vert}}
\def\vhor{\Vc_{\rm hor}}
\def\cvert{\Cc^{\rm vert}}

\def\Rte{{\rm Rt_{ext}}}
\def\Rti{{\rm Rt_{int}}}
\def\XX{{\mathbb X}}
\def\pa{\mathrm{par}}
\def\0{\circ}

\author{Kevin Langlois}
\address{Universit\'e
Grenoble I, Institut Fourier, UMR 5582 CNRS-UJF, BP 74,
38402 St.\ Martin
d'H\`eres c\'edex, France}
\email{kevin.langlois@ujf-grenoble.fr}
\author{Alexander Perepechko}
\address{Department of Mathematics and Mechanics, St. Petersburg State University, Universitetsky pr. 28, Stary Peterhof, 198504, Russia}
\email{perepechko@gmail.com}
\title[Demazure roots and spherical varieties]{Demazure roots and spherical varieties: 
the example of horizontal $\SL_{2}$-actions}

\begin{document}

\begin{abstract}
Let $G$ be a connected reductive group, and let $X$ be an affine $G$-spherical variety.
We show that the classification of  $\mathbb{G}_{a}$-actions on $X$ normalized by $G$ can be reduced to the description of 
quasi-affine homogeneous spaces under the action of a semidirect product $\mathbb{G}_{a}\rtimes G$ with the following property.
The induced $G$-action is spherical and the complement of the open orbit is either empty or a $G$-orbit of codimension one.
 These homogeneous spaces are parametrized by a subset ${\rm Rt}(X)$
of the character lattice $\mathbb{X}(G)$ of $G$, which we call the set of Demazure roots of $X$.
We give a complete description of the set ${\rm Rt}(X)$ when $G$ is a semidirect product
of ${\rm SL}_{2}$ and an algebraic torus; we show particularly that
${\rm Rt}(X)$ can be obtained explicitly as the intersection of a finite union of polyhedra in 
$\mathbb{Q}\otimes_{\mathbb{Z}}\mathbb{X}(G)$  and a sublattice of  $\mathbb{X}(G)$. 
We conjecture that ${\rm Rt}(X)$ can be described in a similar combinatorial way for an arbitrary affine spherical variety $X$. 
\end{abstract}
\maketitle

\
{\footnotesize

\
}

\tableofcontents
\section*{Introduction}
In this article we are interested in the automorphism groups of spherical varieties defined over an algebraically closed field $\KK$ of characteristic zero. Let $G$ be a connected reductive group. 
 Recall that a $G$-spherical variety  is a normal variety endowed with a $G$-action and containing an open orbit of a Borel subgroup.
The Luna--Vust theory describes such varieties with combinatorial objects defined by the geometric structure of the open $G$-orbit. 
Their study has been motivated by numerous important examples (see  \cite{Sat,Dem70, KKMS, MO, Pau, CoPr}) coming in particular from representation theory and enumerative geometry.
The correspondence between these combinatorial objects, called colored fans, and properties of underlying varieties is explained in \cite{Kn}. In the affine case these objects restrict to \emph{colored cones} satisfying certain combinatorial conditions.


To begin studying automorphisms of $G$-spherical varieties it is natural to consider the varieties with a relatively small automorphism group, for instance
those which neutral component is expected to be generated by the $G$-action. The following families of such spherical varieties are studied: 
\begin{enumerate}
\item smooth complete toric varieties \cite{Dem70,AG};
\item flag varieties \cite{Dem77};
\item spherical regular varieties \cite{BiBr};
\item wonderful varieties \cite{Bri07, Pez09};
\end{enumerate}
On the other side, there are families of spherical varieties with quite rich automorphism group, namely such that their automorphism group acts infinitely transitively on the regular locus.  Excluding one-dimensional case, the following families are appropriate:
\begin{enumerate}\setcounter{enumi}{4}
\item affine cones over flag varieties \cite[Section 1]{AKZ};
\item non-degenerate affine toric varieties \cite[Section 2]{AKZ};
\item equivariant $\SL_2$-embeddings considered as spherical varieties under the extension of $\SL_2$ by a torus \cite{Po}, \cite[Chapter III, 4]{Kra84}, \cite{BaHa}, \cite[Section 5.2]{AFKKZ};
\item smooth affine $G$-spherical varieties under a semisimple group $G$ \cite[Section 5.2]{AFKKZ};
\end{enumerate}

Related to this topic, our general aim  is to provide a method of constructing automorphisms of an arbitrary $G$-spherical variety in a combinatorial way.
In this article we study the classification of the $\G_a$-actions on affine $G$-spherical varieties that are normalized by $G$. Restricting to toric varieties, such actions allow studying the automorphism groups in case (1) via total coordinate spaces \cite{Cox95, Cox14}. They also allow establishing the infinite transitivity in case~(6).


In order to explain our results, we introduce the following notation.
Let $X$ be an affine $G$-spherical variety.
We say that two $\G_a$-actions $\phi_1: \G_a\times X\to X$ and $\phi_2\colon \G_a\times X\to X$ are \emph{equivalent}, if there exists a constant $\lambda \in\KK\setminus\{0\}$ such that  $$\phi_1(\lambda\mu,x)=\phi_2(\mu,x) \mbox{ for all } x\in X,\mu\in\G_a.$$ 
In particular, two equivalent actions share the same orbits.
 Denote the $G$-action on $X$ by $(g,x)\mapsto g\cdot x$.
 A $\G_a$-action $$\gamma:\G_a\times X\rightarrow X, \,\,\,(\lambda,x)\mapsto \lambda\ast x$$ 
is called \emph{normalized} by $G$ if there exists a character $\chi\colon G\rightarrow \G_{m}$ such that 
$$
g\cdot (\lambda\ast (g^{-1}\cdot x))=(\lambda\chi^{-1}(g))\ast x
$$
for all $x\in X$, $g\in G$, and  $\lambda\in\G_a$. This character is uniquely defined for a non-trivial action and is  called a \emph{degree} of $\gamma$.
The $\G_a$-action normalized by $G$ naturally defines an action of the semidirect product $\G_a\rtimes_\chi G$ on $X$, see \ref{rappel:norm}.
We denote by ${\rm Rt}(X)$ the set of degrees of normalized $\G_a$-actions on $X$ and call it the \emph{Demazure roots} of $X$.
This definition is coherent with a classical one for toric varieties, see \cite[Section~4.5, p.~571]{Dem70}.  

We distinguish two types of Demazure roots, namely \emph{interior} Demazure roots, whose corresponding $\G_a$-action preserves the open $G$-orbit, and \emph{exterior} ones if not. Obviously $\Rt(X)=\Rte(X)\cup\Rti(X)$, where $\Rte(X)$ and $\Rti(X)$ denote the subsets of exterior and interior roots respectively.

In Theorem~\ref{theoroot} we show that a normalized $\G_a$-action on $X$ is uniquely determined by its Demazure root up to equivalence.  
We also provide a correspondence in terms of Luna--Vust theory between
\begin{itemize}
\item equivalence classes of $\G_a$-actions on $X$ normalized by $G$ and
\item quasi-affine homogeneous spaces
under the action of a semidirect product ${\mathbb{G}_{a}\rtimes_\chi G}$
enjoying the following property:
the induced $G$-action is spherical and the complement of the open orbit is either empty or a $G$-orbit of codimension one.
Furthermore, the character $\chi\in{\rm Rt}(X)$ satisfies a certain explicit condition imposed by $X$.  
\end{itemize}
This correspondence allows us to have a geometrical way to compute the set of Demazure roots of $X$. 
Thus a natural question arises whether $\Rt(X)$ is a combinatorial object. More precisely:
\begin{question*}[\ref{question}]
Can ${\rm Rt}(X)$ be obtained explicitly as the intersection of 
a sublattice of  the character lattice  $\XX(G)$ and
a finite union of polyhedra in 
$\mathbb{Q}\otimes_{\mathbb{Z}}\XX(G)$?
\end{question*}
              
The answer to this question is affirmative if $G$ is an algebraic torus. 
The simplest non-abelian connected reductive group which we can consider is a semidirect product of $\SL_2$ and an algebraic torus $\TT$ of the form $G_e=\SL_2\rtimes_e \TT$. 
It is defined via the multiplication law 
$$(A,t_{1})\cdot (B,t_{2}) = (A\cdot\varepsilon_{e}(t_{1})(B), t_{1}\cdot t_{2}),\mbox{ where } 
\varepsilon_{e}(t)\left(\begin{smallmatrix} a   & b \\ c   &   d \end{smallmatrix} \right)= \left(\begin{smallmatrix}
              a   &  \chi^{-e}(t)b \\
              \chi^{e}(t)c   &   d\end{smallmatrix}\right),$$
where $ \chi^{e}$ denotes a character of $\TT$. We prove that the answer is still affirmative for every affine $G_e$-spherical variety~$X$, see Corollary~\ref{cor:final}.
%

Without loss of generality we may suppose that the action of the torus $\TT\subset G_e$ on $X$ is faithful of complexity one, see Remark~\ref{rem:red}. 
Our main result, Theorem~\ref{th:final}, provides an explicit description of Demazure roots under this assumption. 
Given an affine embedding  $X$ of a $G_e$-homogeneous spherical space $G_e/H$
corresponding to a colored cone $(\Cc,\F)$, we denote by $\F_0$ the subset of colors of $G_e/H$ and
by $\varrho$ the natural map that represents colors as lattice vectors. Furthermore, let $\Gamma$ be the cone generated by $\Cc$ and $\varrho(\F_0)$, denote by $\Rt(\Gamma)$ the set of classical Demazure roots of the toric variety corresponding to $\Gamma$, and by $\Vc$ the cone of $G_e$-invariant valuations. Each root $\ee\in\Rt(\Gamma)$ is associated with a ray of $\Gamma$ denoted by~$\rho_\ee$.
\begin{theorem*}[\ref{th:final}]\label{th:introfinal}
In terms above the set of exterior Demazure roots of $X$ is 
\[
\Rte(X) = \{\ee\in\Rt(\Gamma)\,|\, \rho_{\ee}\in\Vc\cap(-\Vc),\,\,\ee\in\varrho(\F_0)^{\perp}\},
\]
whereas the set of interior Demazure roots is non-empty if and only if $H$ is a subgroup of a maximal torus that is not central in $G_e$ and does not contain a maximal torus of $\SL_2$. In such a case $\F_0$ consists of two different colors, say $D_1,D_2$ and
\[
\Rti(X) = \{\ee\in\Rt(\Gamma)\,|\, \rho_{\ee}\in\Vc\cap(-\Vc),\,\,\{\varrho(D_1)(\ee),\varrho(D_2)(\ee)\}=\{-1,1\}\}.
\]
\end{theorem*}
              
To obtain these results we use \emph{\emph{AL}-colored polyhedral divisors} (by Arzhantsev and Liendo) which provide a combinatorial description of compatible $\SL_2$-actions on affine $\TT$-varieties of complexity at most one, see \cite{AL}. This description allows us to characterize all spherical $G_e$-actions on affine varieties.
In fact, we establish an explicit relation between AL-colorings on $G_e$-spherical varieties and corresponding colored cones in the Luna--Vust theory, 
see Theorem~\ref{theorem:first} and Propositions~\ref{propdivpolversuscol}, \ref{proptoric}.
            

Our situation can be partially interpreted in terms of varieties with a parabolic group action. Consider a connected linear group $P$ with a Levi decomposition $P=L\ltimes U$, where $L$ is a Levi subgroup and $U$ is the unipotent radical. In \cite{Pez14}, they develop the theory of $P/H$-embeddings such that the induced $L$-action is spherical. 
They construct the convex polyhedral cone $V(P/H)$, introduce the set of \emph{spherical roots} $\Sigma(P/H)$ describing that cone, 
 and they expect that $V(P/H)$ is the image of the set of $P$-invariant discrete valuations as in the classical theory of spherical varieties, see  \cite[Sections 5,6]{Pez14}.
In our situation, $L=G$, $U=\G_a$, and the character $\chi$ defining the semidirect product $P=G\ltimes_\chi U$ runs through all the Demazure roots. It would be interesting to study the behaviour of $\Sigma(P/H)$ for different Demazure roots, though we do not address this question in our article.
               
Let us outline the structure of this article.
In Section~\ref{sec:LV} we recall elements of the Luna--Vust theory for the embeddings of spherical homogeneous spaces.
In Section~\ref{sec:2} we briefly explain the Altmann--Hausen construction of polyhedral divisors (see \cite{FZ, AH, Ti2}) for $\TT$-varieties of complexity one and provide several known results concerning normalized $\G_a$-actions (from  
\cite{FZ1, Li, Li2, AL, LL}).
In Section~\ref{sec:3} we establish the connection between AL-colored polyhedral divisors and colored cones.
In Section~\ref{sec:hom} we describe spherical subgroups of $G_e$ in terms of AL-colorings.
In Section~\ref{sec:dem}
we deduce the one-to-one correspondence between equivalence classes of normalized $\G_a$-actions and their Demazure roots in Theorem~\ref{theoroot} and
we pursue Question~\ref{question}, in particular we answer it for affine $G_e$-spherical varieties in Corollary~\ref{cor:final}.

\subsubsection*{Acknowledgements}
We are thankful for valuable suggestions led to Corollary~\ref{fixedpointcor} to Hans\-peter Kraft and other participants of the seminar ``Algebra and Geometry'' at the Mathematisches Institut, Basel Universit\"at. We also thank  Guido Pezzini for kindly communicating to us his preprint \cite{Pez14}.
We are grateful to the referee for his remarks and suggestions that motivated us to develop results stated in Section~\ref{sec:hom}.

The authors express their gratitude to the Institut Fourier and to the Max-Planck Insitute of Mathematics for the hospitality and financial support they provided during the writing of this paper.
The second author acknowledges Saint-Petersburg State Uni\-ver\-sity for a research grant
6.50.22.2014.

\begin{convention}
\label{marker:convention}
Throughout this paper $\KK$ is an algebraically closed field of characteristic zero.
By a variety we mean an integral separated scheme of finite type over $\KK$. 
If $X$ is a variety, then $\KK[X]$ denotes the ring of regular functions of $X$, and $\KK(X)$ denotes its
field of rational functions. A point of $X$ is assumed to be a $\KK$-rational point. 
Furthermore, all algebraic group actions are in particular morphisms of varieties.
For an algebraic group $G$ acting on $X$ the natural action on $\KK[X]$ (resp. $\KK(X)$)
is defined by the formula
\begin{eqnarray*}
(g\cdot f)(x) = f(g^{-1}\cdot x), 
\end{eqnarray*}
where $g\in G$, $f\in \KK[X]$ (resp. $f\in \KK(X)$), and $x\in X$. 
\end{convention}

\section{Preliminaries on Luna--Vust theory}\label{sec:LV}
In this section, we recall some notation from the Luna--Vust theory for
the embeddings of spherical homogeneous spaces. We are mainly interested in the 
case where the embeddings are affine. We refer the reader to \cite{LV, Kn, Pez, Ti3, Pe} for the general theory. 
Let $G$ be a connected reductive linear algebraic group and consider a Borel subgroup $B\subset G$.
We start by introducing classical definitions.

\begin{definition}
A variety with a $G$-action is called \emph{spherical} (or $G$-\emph{spherical}
when one needs to emphasize the acting group) if it is normal and contains an open 
$B$-orbit. A closed subgroup $H\subset G$ is called \emph{spherical} if the homogeneous space $G/H$ is spherical. 
\end{definition}
Let us recall the definition of the scheme of geometric localities associated with a spherical homogeneous space $G/H$ (see \cite[Section 1]{LV}).
This scheme provides us a canonical way to construct a spherical variety by its combinatorial data, see \ref{abs}, thus yielding a bijection
between the colored cones of $G/H$ and embeddings of $G/H$ defined as follows. 

\begin{rappel}
\label{marker:emb}
A subalgebra of $\KK(G/H)$ is \emph{affine} 
if it is finitely generated and if $\KK(G/H)$ is equal to its field of fractions.
A \emph{geometric locality} is a local algebra obtained by the localization
of an affine subalgebra of $\KK(G/H)$ and having residue field equal to $\KK$. 
The set $\Sch$ of geometric localities is naturally endowed with
a structure of $\KK$-scheme, where the spectrum of each affine subalgebra of $\KK(G/H)$
is an open subset. The group $G$ acts on $\Sch$ as an abstract group; this action
is induced by the $G$-action on $\KK(G/H)$.
We denote by $\Emb$ the maximal $G$-stable normal open subset for which the $G$-action is regular 
(see \cite[Propositions~1.2 and 1.4]{LV}).
An \emph{embedding} of $G/H$ is a $G$-stable separated Noetherian open subset of $\Emb$.
If $G$ is an algebraic torus, then 
an embedding of $G$ is called a \emph{toric variety}. 
\end{rappel}
The toric varieties are described by 
objects from the convex geometry called fans, e.g. see \cite{KKMS, MO, Oda, Ful}. The Luna--Vust theory is a
generalization of this description for embeddings of homogeneous spaces.  
In the sequel, we denote by $X$ an affine embedding of a spherical homogeneous space $G/H$. 
We choose $x\in X$ so that $B.x$ is the open $B$-orbit.
\begin{definition}
Let $K$ be an algebraic group acting on a variety $Z$. A $K$-\emph{divisor}
of $Z$ is a $K$-stable prime divisor on $Z$. In our situation, a $B$-divisor 
of the spherical variety $X$ that is not $G$-stable is called a \emph{color}. 
\end{definition}
Identifying $G.x$ with $G/H$
we may view a color as the closure in $X$ of a $B$-divisor of $G/H$. Thus, colors are determined by the homogeneous 
space $G/H$.
\begin{rappel}
\label{marker:val1}
The algebra of rational functions $\KK(X)=\KK(G/H)$ is endowed with a natural linear $B$-action.
Consider the lattice $L$ of $B$-weights of $\KK(X)$ for this action. For a $B$-eigenvector $f$ in $\KK(X)$ we denote by 
$\chi_f$ the corresponding weight. Since 
the quotient of two eigenvectors of the same weight is $B$-invariant and therefore constant on 
the open orbit $B.x$, the $B$-eigenspaces in $\KK(X)$ are of dimension one, 
hence the weight $\chi_f$ defines $f$ up to a scalar. 
Summarizing, we have the exact sequence
$$
1\to \K^{*}\to \KK(X)^{(B)}\to L\to 0,
$$
where $\KK(X)^{(B)}$ denotes the multiplicative subgroup of $B$-eigenvectors of $\KK(X)^{*}$. The epimorphism $\KK(X)^{(B)}\to L$
is given by $f\mapsto \chi_{f}$.
\end{rappel}
\begin{rappel}
\label{marker:val2}
Recall that a \emph{discrete valuation} on $\KK(X)$ is a map $v:\KK(X)^{\star}\rightarrow \QQ$
such that
\begin{itemize}
\item $v(f+g)\geq \min\{v(f), v(g)\}$ for all $f,g\in \KK(X)^{\star}$ satisfying $f+g\in\KK(X)^{\star}$;
\item $v$ is a group morphism from $(\KK(X)^{\star},\times)$ to $(\QQ, +)$, whose image is a subgroup with one generator;
\item the subgroup $\KK^{\star}$ is contained in the kernel of the morphism $v$.
\end{itemize}
Let $V=\Hom(L,\ZZ)$
be the dual lattice of $L$ and let $V_{\QQ} = \QQ\otimes_{\mathbb{Z}}V$ be the associated $\QQ$-vector space. 
For every discrete valuation $v$ of $\KK(X)$ we associate a vector $\vrho(v)\in V_\QQ$ by letting
$$\vrho(v)(\chi_f) =  \langle \chi_{f}, \vrho(v) \rangle = v(f)\mbox{ for any } f\in\KK(X)^{(B)}.$$
Since $X$ is normal, each prime divisor $D\subset X$ naturally defines a discrete valuation 
$v_D$ on $\K(X)$. So, we may identify any subset of $B$-divisors  
with the corresponding set of discrete valuations, which we denote by the same letter.
\end{rappel}
\begin{rappel}
A discrete valuation $v$ on $\KK(X)$ is said to be $G$-\emph{invariant} if every fiber of the map $v$
is $G$-stable, for the natural linear $G$-action on $\KK(X)$. It is
well known that the restriction of $\vrho$ to the subset of discrete $G$-invariant valuations is injective \cite[Section 7.4]{LV}. 
Furthermore, the image of this restriction is a polyhedral cone $\Vc$ of $V_{\mathbb{Q}}$ (see \cite[Corollary 3.2]{BP}),
called the \emph{valuation cone} of $G/H$. Thus, we will say that $V$ is the \emph{valuation lattice}.
Recall that the spherical variety $X$ is said to be \emph{horospherical} if the isotropy group of a general point contains
a maximal unipotent subgroup of $G$. It is well-known that $X$ is horospherical if and only $\Vc = V_{\QQ}$
(see \cite[Corollary 6.2]{Kn}). 
\end{rappel}
Let us introduce the classical notion of colored cone \cite[p. 230]{LV}.
In this definition, an element of $L$ is seen as a linear form on $V_{\QQ}$.
\begin{definition}
\label{definition:conecol}
A \emph{colored cone} of the homogeneous space $G/H$ is a pair $(\Cc, \F),$
where 
\begin{itemize}
\item $\F$ is a subset of colors of $G/H$ such that $0\notin\vrho(\F)$;
\item  $\Cc\subset V_\QQ$ is a strongly convex polyhedral cone generated by the union of $\vrho(\F)$ and of a
finite subset of $\Vc$;
\item the intersection of the relative interior of $\Cc$ with the cone $\Vc$ is non-empty.
\end{itemize}
Moreover, a colored cone $(\Cc, \F)$ will be called \emph{affine} if there exists a lattice vector $m\in L$ 
satisfying the following conditions:
\begin{itemize}
\item $m$ is non-positive on $\Vc$;
\item $m$ is zero on $\Cc$;
\item $m$ is positive on all colors of $G/H$ not belonging to $\F$.
\end{itemize}
\end{definition}
There is a natural way to build an affine colored cone of $G/H$ from the embedding $X$, see \cite[Section 2]{Kn}.
In the next paragraph, we explain this construction.

\begin{rappel}
\label{marker:col}
Note that the open orbit $B.x$ is affine (see \cite[Theorem 3.5]{Ti3}), and so its
complement in $X$ is a union of $B$-divisors. So, $X$ has a finite number
of $B$-divisors. We denote by $\cP$ the set of $G$-divisors of $X$ identified
with a subset of $V_{\mathbb{Q}}$. Moreover, since $X$ is affine, $X$ has a unique closed
$G$-orbit $Y$, which is contained in the closure of all the $G$-orbits of $X$.
We denote by $\F_{Y}$ the set of colors of $X$ containing $Y$. The \emph{associated colored
cone} of $X$ is the pair $(\Cc, \F_{Y})$, where $\Cc\subset V_\QQ$ is the cone
generated by the union of $\cP$ and of the subset $\vrho(\F_Y)$. According to 
\cite[Section 2]{Kn} and \cite[Theorem 6.7]{Kn} the pair $(\Cc, \F_{Y})$ 
is in fact an affine colored cone of $G/H$.
Note that the embedding $X$ 
is called \emph{toroidal} if $\mathcal{F}_{Y} = \emptyset$.
\end{rappel}
Conversely, to a colored cone $(\Cc, \F)$ of $G/H$ one can associate an embedding.
\begin{rappel}
\label{abs}
Let $(\Cc, \F)$ be a colored cone of $G/H$. 
For a discrete valuation $v$ of $\KK(G/H)$ we write $\Oo_{v}$ the
corresponding local algebra. Let $X_{B}$ be the open $B$-orbit in $G/H$.
Considering a finite subset $\cP\subset \Vc$ such that $\cP\cup \vrho(\F)$ generates
the cone $\Cc$, we define the algebra
\begin{eqnarray*}
\Ac = \KK[X_{B}]\cap \bigcap_{v\in\cP}\Oo_{v}\cap \bigcap_{D\in\F}\Oo_{v_{D}}. 
\end{eqnarray*}
By \cite[Section 8]{LV}, $\Ac$ is a normal affine subalgebra of $\KK(G/H)$.
The variety $X_{0} = {\rm Spec}\, \Ac$ does not depend on the choice of the set $\mathcal{P}$.
Furthermore, $X_{0}$ is an open subset of $\Emb$ (compare with \cite[Proposition 8.10]{LV}).
The associated spherical embedding of $(\Cc, \F)$ is the subscheme $X = G\cdot X_{0}$ \cite[Proposition 1.5]{LV}.
 This construction is also explained in \cite[Chapter 3, Section 13]{Ti3}.
\end{rappel}
More precisely, we have the following classical result (see \cite[Theorem 6.7]{Kn}).
\begin{theorem}
The map defined by sending an affine embedding $X$ of $G/H$ to its associated colored cone
$(\Cc,\F_{Y})$ is a bijection between
\begin{itemize}
\item the set of affine embeddings of $G/H$ and
\item the set of affine colored cones of $G/H$.
\end{itemize}   
\end{theorem}
The next well-known lemma will be useful afterwards (see the proof of \cite[6.7]{Kn}). 
\begin{lemme}
\label{lemma:affcone}
Let $X$ be an affine embedding of $G/H$ with associated colored cone $(\Cc, \F_{Y})$.
Let $\F_{0}$ be the set of colors of $G/H$. Then the cone $\Gamma$ dual to the cone generated by the $B$-weights of $\KK[X]$ is 
the polyhedral cone in $V_{\QQ}$ generated by the subset $\Cc\cup\varrho(\F_{0})$. In addition,
$\Gamma$ is strongly convex. 
\end{lemme}
\begin{proof}
Let $U$ be the unipotent part of $B$. Since $X$ is affine, the field of fractions
of $\KK[X]^{U}$ is equal to $\KK(X)^{U}$ (see \cite[Lemma D.7]{Ti3}). 
Consequently, every $B$-eigenvector of $\KK(X)$
is the quotient of two eigenvectors of $\KK[X]$. This shows that $\Gamma$ is a strongly convex cone.
The rest of the proof is straightforward. 
\end{proof}


\section{Preliminaries on polyhedral divisors and normalized $\G_{a}$-actions}
\label{sec:2}
In this section, we recall some basic facts from the theory
of polyhedral divisors for the case of complexity one \cite{FZ, AH, Ti2}.
We present also a brief survey of known results concerning the classification
of normalized $\G_{a}$-actions on affine $\TT$-varieties of complexity one, see \cite{FZ1, Li, Li2, AL, LL}.

\subsection{Polyhedral divisors and affine $\TT$-varieties of complexity one}
\

We denote by $\TT$ an algebraic torus of 
dimension $n$. We fix a lattice $M$ isomorphic to  the character group of $\TT$ via $m\mapsto \chi^{m}$. 

\begin{definition}
The \emph{complexity} of a $\mathbb{T}$-action on a variety $X$ is the transcendence degree over $\KK$ of the field
extension $\KK(X)^{\TT}$. By a result of Rosenlicht \cite{Ro}, the complexity is also the codimension of 
a general orbit in $X$.
In particular, if the action is faithful, then the complexity is equal to 
${\rm dim}\, X - {\rm dim}\,\TT$. A \emph{$\mathbb{T}$-variety} is a normal variety endowed with 
a faithful $\mathbb{T}$-action.
\end{definition}
Note that having a $\TT$-action on an affine variety $X = {\rm Spec}\,A$ is equivalent to
endowing $A$ with an $M$-grading. If $\TT$ acts on $X$ and $m\in M$, then we define the $m$-th graded component
of $A$ by letting 
$$A_{m} = \{f\in A\,\,|\,\, t\cdot f = \chi^{m}(t)f\,\mbox{ for any }t\in\TT\}.$$
In \cite{AH}, a combinatorial description of $M$-graded algebras
corresponding to affine $\TT$-varieties is given in terms of polyhedral divisors.
We recall this construction in the setting of the complexity one. We also introduce
other notation from convex geometry which will be useful in the sequel.
\begin{rappel}
Let $N = {\rm Hom}(M,\mathbb{Z})$ be the dual lattice of $M$.
Denote by $M_{\mathbb{Q}} = \mathbb{Q}\otimes_{\mathbb{Z}}M$ and
$N_{\mathbb{Q}} = \mathbb{Q}\otimes_{\mathbb{Z}}N$ the associated vector spaces.
For a polyhedral cone $\tau\subset N_\QQ$ the symbols $\tau^{\vee}$, $\linea(\tau)$, and $\relint(\tau)$ stand respectively for 
the dual cone in $M_{\mathbb{Q}}$, the linear part $(-\tau)\cap \tau$, 
and the relative interior of $\tau$.
Considering $v\in N_{\QQ}$ we may write $v^{\perp} = \{m\in M_{\QQ}\,|\, v(m) = 0\}$
for its orthogonal space and denote $\mu(v)=\inf\{r\in\ZZ_{>0}\;|\;rv\in N\}$.

\label{marker:defpol}
For every polyhedron $\Delta$ of $N_\QQ$ we denote by
$\Delta(r)$ the set of its $r$-dimensional faces.
Let $\sigma\subset N_\QQ$ be a strongly convex polyhedral cone.
A subset $\Delta\subset N_{\QQ}$ is called a \emph{$\sigma$-polyhedron} if $\Delta$
is a Minkowski sum $Q+\sigma$ for a polytope\footnote{We recall that 
a polytope in a $\QQ$-vector space is the convex hull of a non-empty
finite subset.} $Q\subset N_{\mathbb{Q}}$.

A \emph{$\sigma$-polyhedral divisor} over a curve $C$ is a formal sum 
$\D=\sum_{z\in C}\Delta_z\cdot  z$, where each $\Delta_z$ is a $\sigma$-polyhedron, with the condition that 
$\Delta_z=\sigma$ for all but finitely many $z\in C$. The finite subset $\{z\in C\;|\;\Delta_z\neq\sigma\}$ is called the \emph{support} of $\D$.
 The \emph{degree} of $\D$, denoted by $\deg\D$,
is the Minkowski sum of all the coefficients of $\D$, which is a $\sigma$-polyhedron (see \cite[2.12]{AH}). 
For a vector $m\in M_\QQ$ we let
$$\D(m)=\sum_{z\in C}\min_{v\in\Delta_z(0)} \{v(m)\}\cdot z,$$
which is a $\QQ$-divisor on $C$. 
If $C$ is a normal curve, then the associated $M$-graded algebra
of $\D$ is 
$$
A[C,\D] =  \bigoplus_{m\in\sigma^{\vee}\cap M} H^{0}(C, \Oc_C(\lfloor\D(m)\rfloor))\chi^{m}.
$$
The multiplication on $A[C,\D]$ is defined on homogeneous elements via the maps
$$H^{0}(C, \Oc_C(\lfloor\D(m)\rfloor))\chi^{m}\times H^{0}(C, \Oc_C(\lfloor\D(m')\rfloor))\chi^{m'}\rightarrow 
H^{0}(C, \Oc_C(\lfloor\D(m + m')\rfloor))\chi^{m+m'}$$
sending $(f_{1},f_{2})$ to $f_{1}\cdot f_{2}$. 
The polyhedral divisor 
$\D$ is called \emph{linear} on a polyhedral cone $\omega\subset M_{\QQ}$
if $$\D(m + m') = \D(m) + \D(m')\mbox{ for all } m,m'\in\omega.$$ We denote by 
$\Lambda(\D)$ the coarsest quasifan of $\sigma^{\vee}$ where $\D$ is linear on each cone (see \cite[Section 1]{AH}
for details).  
\end{rappel}
\begin{definition}
A $\sigma$-polyhedral divisor $\D$ on a normal curve $C$ is said to be \emph{proper} if $\D(m)$ 
is semi-ample for any $m\in \sigma^{\vee}$ and big for any $m\in\relint(\sigma^{\vee})$.
This condition is automatically fulfilled in case of an affine curve $C$. Otherwise, if $C$ is projective, 
then the properness condition can be split in two parts:
\begin{itemize}
\item $\deg\D\subsetneq \sigma$;
\item for any $m\in \sigma^{\vee}$ such that $\min_{v\in\deg\D}\{v(m)\} = 0$,
$m$ belongs to the boundary of $\sigma^{\vee}$ and $\D(rm)$ is a principal divisor for some $r\in \ZZ_{>0}$.
\end{itemize}
\end{definition}
In this paper we deal with proper polyhedral divisors over $\AA^{1}$ or $\PP^{1}$. In this case,
we have a simple characterization of the properness: a $\sigma$-polyhedral divisor over $\PP^{1}$ 
is proper if and only if $\deg\D\subsetneq \sigma$. 
\begin{rappel}
If $\D$ is a proper polyhedral divisor over a normal curve $C$, then $X=\Spec A[C,\D]$ is an affine $\TT$-variety 
of complexity one. The cone $\sigma^{\vee}$ is the weight cone of the $M$-graded algebra $\KK[X]$. The curve $C$ 
is a rational quotient of $X$, i.e. 
$\KK(X)^{\TT} = \KK(C)$. 
Conversely, every affine $\TT$-variety of complexity one is obtained in this way \cite[Theorem 3.1]{AH}. 
For the uniqueness of such 
representation, see \cite[Section 8]{AH}.
\end{rappel}
Let $\mathfrak{D}$ be a proper $\sigma$-polyhedral divisor over a normal curve $C$.
Consider the associated algebra $A = A[C,\D]$ and let $X = {\rm Spec}\, A$. In the following paragraph
we explain a classification of $\TT$-divisors in terms
of the pair $(C,\D)$ corresponding to $X$. See  \cite[Section 7]{AH}, \cite{Ti2}, and \cite{PeSu} for further details.
\begin{rappel}\label{rappel:div-hor-vert}
Let $\sistar$ be the following subset of $\sigma(1)$. If $C$ is affine,
then we let $\sistar = \sigma(1)$. Otherwise, $C$ is projective and $\sistar$
is the set of $1$-dimensional faces of $\sigma$ that do not meet  $\deg\D$. For every element $\rho\in\sigma(1)$ we denote by the same letter $\rho$
its primitive lattice generator. There are two types of $\TT$-divisors in the variety $X$.
\begin{itemize}
\item \emph{Horizontal} divisors denoted by $D_\rho$ for each $\rho\in \sistar$. The corresponding
$M$-graded ideal of such divisor $D_{\rho}$ is
$$
\I(D_\rho)=\bigoplus_{m\in(\sigma^{\vee}\setminus\rho^{\perp})\cap M} A_m\chi^{m}, \mbox{ where } A_m=H^0(C, \Oc_C(\lfloor\D(m)\rfloor)).
$$
A $\TT$-divisor $D$ is horizontal if and only if
the induced $\TT$-action on $D$ is of complexity one.
\item \emph{Vertical} divisors denoted by $D_{(z,v)}$, where $z\in C$ and $v\in\Delta_z(0)$ is a vertex of $\Delta_z$. 
The associated $M$-graded ideal is 
$$
\I(D_{(z,v)})=\bigoplus_{m\in\sigma^{\vee}\cap M} (A_m\cap \{f\in\KK(C)\;|\;\ord_z f > -v(m)\})\chi^{m}.
$$
A $\TT$-divisor $D$ is vertical if and only if 
the induced $\TT$-action on $D$ is of complexity zero.  
\end{itemize}

Moreover, if $f\chi^{m}\in A$ is homogeneous of degree $m$, 
then the corresponding principal divisor is given by the formula
$$
\div(f\chi^{m})= \div(f\chi^{m})_{\rm hor} + \div(f\chi^{m})_{\rm ver} 
$$
where 
$$\div(f\chi^{m})_{\rm hor} = \sum_{\rho\in\sistar}\rho(m)\cdot D_\rho, $$
$$\div(f\chi^{m})_{\rm ver}  = \sum_{z\in C}\sum_{v\in\Delta_z(0)}\mu(v) \left(v(m)+\ord_z f\right)\cdot D_{(z,v)},$$
see \cite[Proposition 3.14]{PeSu}. The divisor $\div(f\chi^{m})_{\rm hor}$
(resp. $\div(f\chi^{m})_{\rm ver}$) will be called the \emph{horizontal part} (resp. \emph{vertical part}) of $\div(f\chi^{m})$.
\end{rappel}
\subsection{$\G_{a}$-actions on affine $\TT$-varieties of complexity one}
\

In this subsection, we review results in \cite{Li} concerning the classification 
of normalized additive group actions on affine $\TT$-varieties of complexity one.
\begin{rappel}\label{rappel:norm}
Let $Z$ be a variety. Let $$K\times Z\rightarrow Z,\,\,\,(g,x)\mapsto g\cdot x$$ be an action of an algebraic 
group $K$ on $Z$. 
%
Consider the algebraic group 
$\G_{a}\rtimes_{\chi}K$, where the underlying variety is $\G_a\times K$ and the 
group law is given by
$$(\lambda_{1}, g_{1})\cdot (\lambda_{2}, g_{2}) = (\lambda_{1} + \chi^{-1}(g_{1})\lambda_{2}, g_{1}\cdot g_{2})$$
for $\lambda_{i}\in \G_{a}$ and $g_{i}\in K$; this algebraic group is the semidirect product
of $\G_{a}$ and $K$ associated with the character $\chi$.
Moreover, a $K$-action on $Z$ and an additive group action normalized by $K$
of degree $\chi$ naturally generate the action of $\G_a\rtimes_{\chi} K$ on $Z$
 via the formula
$$(\lambda, g)\cdot x = \lambda\ast (g\cdot x),$$ where $(\lambda, g)\in \G_{a}\rtimes_{\chi}K$
and $x\in X$. If $K$ is the torus $\TT$ and $e\in M$ is lattice vector, then we denote $\G_{a}\rtimes_{e}\TT = \G_{a}\rtimes_{\chi^{e}}\TT$.  
\end{rappel}
The following definition is introduced in \cite{FZ1} to describe normal quasi-homogeneous affine surfaces and
extended in \cite[Section 1.2]{Li} to varieties with a $\TT$-action.  
\begin{definition}
Let $Z$ be a variety endowed with a $\mathbb{T}$-action.
An additive group action on $Z$ normalized by $\TT$ is called \emph{vertical}
if a general $\G_a$-orbit of $Z$ is contained in the closure of a $\TT$-orbit. 
Otherwise, the action is called \emph{horizontal}.
\end{definition}
\begin{rappel}
\label{marker:Ga}
For an affine variety $Z$, we recall that a locally nilpotent derivation (called LND for brevity) on $\KK[Z]$ is a $\KK$-derivation $\partial:\KK[Z]\rightarrow \KK[Z]$ 
such that for any $f\in\KK[Z]$ there holds $\partial^{i}(f) = 0$ for a sufficiently large $i\in\ZZ_{>0}$.
There is a well-known one-to-one correspondence between the $\G_a$-actions on $Z$ 
and the LNDs on $\K[Z]$. 
Indeed, let $\KK[\G_a] = \KK[\lambda]$ for a variable $\lambda$ over $\KK$. Then, 
given an LND $\partial$ on $\K[Z]$, the morphism
$$\KK[Z]\to\KK[Z]\otimes_\KK\KK[\lambda],\quad f\mapsto\sum_{i\ge0}\frac{\partial^{i}(f)}{i!}\otimes\lambda^{i},$$ 
defines a $\G_a$-action $\gamma_{\partial}$ on $Z$, and vice-versa. See \cite{Fre} for details concerning 
the theory of LNDs. Now assume that $Z$ is endowed with a $\TT$-action.
It is also known that $\G_a$-actions on $Z$ normalized by $\TT$ of degree $\chi^{e}$ correspond 
to homogeneous LNDs of degree $e$ with respect to the $M$-grading on $\K[Z]$ (see the proof
of \cite[Lemma 2.2]{FZ1}). 
\end{rappel}
For the next paragraphs, we keep the same notation as above for the $\sigma$-polyhedral divisor $\D$. We let
$X = {\rm Spec}\, A[C,\D]$ be the associated affine $\TT$-variety of complexity one.
\begin{rappel}
Let $\partial$ be a homogeneous LND on $\KK[X]$. We have the following characterization.
\begin{itemize}
\item $\gamma_{\partial}$ is vertical if and only if $\KK(C)\subset \KK(X)^{\partial}$, where $\KK(X)^{\partial}\subset \KK(X)$ is the 
field of invariants.
\item $\gamma_{\partial}$ is horizontal if and only if $\KK(C)\cap\KK(X)^{\partial}=\K$. In other words,
the algebraic group $\G_a\rtimes_e \TT$ 
acts on $X$ with an open orbit.
\end{itemize} 
We call the corresponding LNDs \emph{vertical} and \emph{horizontal} respectively. 
The set of non-zero vertical (resp. horizontal) LNDs on $\KK[X]$ is denoted by $\LND_{ver}(X)$
(resp. $\LND_{hor}(X)$).
\end{rappel}
The following paragraph presents the classical description of vertical homogeneous LNDs in 
complexity one. See \cite{Li2} for the general case.
\begin{rappel}
\label{marker: Demazure}
Recall from \cite[Section~4.5, p.~571]{Dem70} that a vector $e\in M$ is a \emph{Demazure root} of the cone $\sigma$ with distinguished ray $\rho_{e}\in\sigma(1)$ 
if $\rho_e(e)=-1$ and $\rho(e)\ge0$ for any $\rho\in\sigma(1)\setminus\{\rho_e\}$.
We denote by $\Rt(\sigma)$ the set of all roots of $\sigma$. If $e\in \Rt(\sigma)$, then we let 
$$
\Phi_e^*=H^0(C,\Oc_C(\lfloor\D(e)\rfloor))\setminus\{0\}.
$$
Furthermore, $\Phi_e^*\neq\emptyset$ if and only if $\rho_{e}\in\sistar$, see \cite[Corollary 3.13]{Li}. Assume
that $\Phi_e^*$ has an element $\phi$. We define a homogeneous LND $\partial_{e,\phi}$ of vertical type on the 
algebra $A = A[C, \D]$ by the formula
$$
\partial_{e, \phi}(f\chi^m)=
\rho_e(m)\phi f\chi^{m+e}\raisebox{-1.5pt},
$$
where $f\chi^{m}\in A$ is an homogeneous element of degree $m$, see \cite[Lemma 3.6]{Li}. 
More precisely, let
$$\Rt^{\ast}(\sigma) = \{e\in\Rt(\sigma)\,|\, \rho_{e}\in\sistar\}.$$
Considering the set $$\Phi=\{(e,\phi)\;|\;e\in \Rt^{\ast}(\sigma),\phi\in\Phi_e^*\},  
$$ the map $\Phi\to \LND_{ver}(A),\,(e,\phi)\mapsto\partial_{e,\phi}$ is a bijection \cite[Theorem 3.8]{Li}.
\end{rappel}
In the sequel, we consider the case of horizontal homogeneous LNDs in complexity one. Note that if $X = {\rm Spec}\, A[C,\D]$ admits a normalized 
$\G_a$-action of horizontal type, then either $C=\AA^1$ or $C=\PP^1$ (see \cite[3.16]{Li}). 
The next paragraph introduces combinatorial
objects that describe homogeneous LNDs of horizontal type on $A = A[C,\D]$ 
(see \cite[Section 3.2]{Li}, \cite[Section 1.4]{AL}, and \cite[Section 5]{LL}).
Note that our terminology will be slightly different.  
%
\begin{rappel} 
Let $\D$ be a proper $\sigma$-polyhedral divisor over $\AA^1$.
An  \emph{\emph(affine\emph) \emph{AL}-coloring}\footnote{named by I. Arzhantsev and A. Liendo. The sign $\bullet$ will be substituted by $-$ and $+$ later on, when we consider different AL-colorings of a polyhedral divisor.} on the polyhedral divisor $\D$ is 
a triple $\dcol=(\D,\{v_z^\bullet\;|\;z\in \AA^{1}\}, z_{0}),$ where $z_{0}\in \AA^{1}$,
satisfying the following conditions:
\begin{itemize}
\item each $v_z^\bullet$ is a vertex of $\Delta_{z}$;
\item $v_{\deg}^\bullet\coloneq\sum_{z\in \AA^{1}}v_z^\bullet$ is a vertex of the polyhedron $\deg\D = \sum_{z\in \AA^{1}}\Delta_{z}$;
\item $v_z^\bullet\in N$ for any $z\in \AA^{1}\setminus\{z_0\}$.
\end{itemize}
The elements $v_{z}^\bullet$ for $z$ in the support of $\D$ are called the AL\emph{-colors} of $\dcol$. 
The point $z_{0}$ is called the \emph{marked point} of $\dcol$. 
The \emph{support} of $\dcol$ is defined as the union of the support of $\D$ and of the marked point $z_0$.  
The polyhedron
$$\omega_{\bullet} = \Cone(\deg\D-v_{\deg}^\bullet)$$
is called the \emph{associated cone} of $\dcol$. 

Now let $\D$ be a proper $\sigma$-polyhedral divisor over $\PP^1$.
A \emph{\emph(projective\emph) \emph{AL}-coloring} on the polyhedral divisor $\D$ is 
a quadruple $\dcol=(\D,\{v_z^\bullet\;|\;z\in \AA^{1}\}, z_{0}, z_\infty)$, where $z_\infty\in \PP^1$, $\AA^1=\PP^1\setminus\{z_\infty\}$, and the triple $\dcol^{\aff}=(\D|_{\AA^{1}},\{v_z^\bullet\;|\;z\in \AA^{1}\}, z_{0})$ is an affine AL-coloring.
We define the AL-colors, the marked point, the associated cone, and the support of $\dcol$ as those of $\dcol^{\aff}$. 
We also call $z_\infty$ the \emph{point at the infinity} of $\dcol$.

\end{rappel}
\begin{remark}\label{omega-sens}
The definition of $\omega_{\bullet}$ is motivated by the following relation:
\begin{equation}
\D(m)|_{\AA^{1}}=\sum_{z\in \AA^{1}} v_z^{\bullet}(m)\cdot z \mbox{ if and only if } m\in\omega^\vee_{\bullet}.
\end{equation}
\end{remark}
\begin{rappel}
Since
we will consider several colorings on a 
same polyhedral divisor, it is convenient to represent them by a weighted graph.
Let $\{z_{0},\ldots , z_{r}\}\subset \AA^{1}$ be the support of $\dcol$, where $z_{0}$
is the marked point. A representation of $\dcol$ will be drawn as 
$$\dots \lin \cou{v_{z_{0}}^\bullet}{z_0} \lin \cou{v_{z_{1}}^\bullet}{z_1} \lin \dots \lin \cou{v_{z_{r}}^\bullet}{z_r} \lin \dots,$$
where $v_{z_{0}}^\bullet,\ldots, v_{z_{r}}^\bullet$ are the corresponding AL-colors of $\dcol$.
\end{rappel}

   
Let us introduce other objects attached to a coloring of $\D$. In the following definition,
we keep the same notation as above for the AL-coloring
$\dcol$.
\begin{definition}
\label{marker:defcohe}
Let $\tilde{\omega}_{\bullet}\subset (N\oplus \mathbb{Z})_{\QQ}$ be the 
cone generated by $(\omega_{\bullet}, 0)$ and $(v_{z_{0}}^\bullet, 1)$ if $C = \AA^{1}$, and by $(\omega_{\bullet}, 0), (v_{z_{0}}^\bullet,1)$,
and $(\Delta_{z_{\infty}}+ v_{\rm{deg}}^\bullet -v_{z_{0}}^\bullet,-1)$ if $C = \PP^{1}$. 
The polyhedral cone $\tilde{\omega}_{\bullet}$ will be called the \emph{domain} of $\dcol$.
Let $e\in M$.
A pair $(\dcol,e)$ is said to be \emph{coherent} 
if the following conditions are satisfied.
\begin{itemize}
\item Let $s= -\frac{1}{d(e)} - v^\bullet_{z_0}(e)$, where $d(e)=\inf\{d\in\ZZ_{>0}\;|\; dv^\bullet_{z_0}(e)\in\ZZ\}$. The vector $(e,s)$ is a Demazure root of $\tilde{\omega}_{\bullet}$;
\item $v(e)\ge1+v_z^\bullet(e)$ for any $z\in \AA^{1}\setminus\{z_0\}$ and for any $v\in\Delta_z(0)\setminus\{v_z^\bullet\}$;
\item $v(e)\ge-s$ for any $v\in\Delta_{z_0}\setminus\{v_{z_0}^\bullet\}$;
\item If $C=\PP^1$, then  $v(e)\ge-\frac{1}{d(e)}-v_{\deg}^\bullet(e)$ for any $v\in\Delta_{z_\infty}(0)$.
\end{itemize}
\end{definition}
The next theorem provides a description of normalized $\G_a$-actions 
of horizontal type on the $\TT$-variety
$X = {\rm Spec}\, A[C,\D]$
in terms of coherent pairs \cite[Theorem 1.10]{AL}. We 
fix a variable $t$ over $\KK$ such that $z_0=\{t=0\}$ and $z_\infty=\{t=\infty\}$ (if exists), in particular $\KK[\AA^{1}] = \KK[t]$. 
\begin{theorem}
\label{theorem:hor}
 Let $\D$ be a proper $\sigma$-polyhedral divisor over $C = \AA^{1}$ or $C = \PP^{1}$ and let $A = A[C,\D]$. 

i) Consider a triplet $(\lambda,\dcol,e)$, 
where $\lambda\in\KK^{\ast}$, $\dcol$ is an ${\rm AL}$-coloring of $\D$, and $(\dcol,e)$ is a coherent pair. 
Assume without loss of generality that $v_z^\bullet=0$ for any $z\in \AA^{1}\setminus\{0\}$.
Each such triplet $(\lambda,\dcol,e)$ defines a horizontal \emph{LND} $\partial$ on $A$ via the formula
$$
\partial(t^r\chi^m)=\lambda\cdot d(e)\cdot (v_0^\bullet(m)+r)\cdot t^{r+s}\chi^{m+e} \,\,\,\,{\rm for\,\,all}\,\,\,(m,r)\in M\oplus\mathbb{Z},
$$
where $s = -\frac{1}{d(e)} - v_{0}^\bullet(e)$.
The kernel of the {\rm LND} $\partial$ corresponding to the triplet $(\lambda,\dcol,e)$ is 
$$
\ker\,\partial=\bigoplus_{m\in\omega_{\bullet}^{\vee}\cap L} \KK\phi_m\chi^{m},
$$
where $L=\{m\in M\;|\; v_{0}^\bullet(m)\in\ZZ\}$, $\phi_m\in\KK(C)^{\ast}$ such that $(\div(\phi_m)+\D(m))|_{\AA^{1}}=0$,
and $\omega_{\bullet}$ is the associated cone of $\dcol$. Moreover, 
$\omega_{\bullet}^\vee$ is a maximal cone of $\Lambda(\D|_{\AA^{1}})$ {\rm (}see \ref{marker:defpol}{\rm )}.

ii) All horizontal \emph{LND}s on $A$ are obtained in this way.
\end{theorem}
\begin{remark}\label{rem:general}
More generally, if $z_0$ is arbitrary, $z_\infty=\infty$, and $\dcol$ be represented by
$$\dots \lin \cou{v_{z_{0}}^\bullet}{z_0} \lin \cou{v_{z_{1}}^\bullet}{z_1} \lin \dots \lin \cou{v_{z_{k}}^\bullet}{z_k} \lin \dots$$
then the LND $\partial_\bullet$ is given by the formula
\begin{align*}
\partial_\bullet((t-z_0)^{r}\xi_m\chi^{m})=d(e)(v_{z_0}^\bullet(m)+r)(t-z_0)^{r-s}\xi_{m+e}\chi^{m+e},\\
\end{align*}
where
$$ \xi_m=\prod_{i=1}^{k}(t-z_i)^{-v_{z_i}^\bullet(m)}
\mbox{ and } s=-\frac{1}{d(e)}-v_{z_0}^\bullet(e).$$

\end{remark}

\subsection{Horizontal $\SL_2$-actions}
\

In this subsection, we recall generalities on compatible $\SL_{2}$-actions (see \cite[Section 2]{AL})
of horizontal type for affine $\TT$-varieties of complexity one. 
\begin{rappel}
\label{marker:semiprod}
We make the convention that $\SL_{2}$ is the algebraic 
group of $2\times 2$ matrices over $\KK$ with determinant one.
For every $e\in M$ we define the reductive group $G_{e} = \SL_{2}\rtimes_e \TT$ as follows. 
The underlying variety is $\SL_{2}\times\TT$ and the multiplication law is given by the relation
$$(A,t_{1})\cdot (B,t_{2}) = (A\cdot\varepsilon_{e}(t_{1})(B), t_{1}\cdot t_{2}),$$
where $(A,t_{1}), (B,t_{2})\in G_{e}$ and  the function $\varepsilon_{e}$ satisfies
$$\varepsilon_{e}(t)\begin{pmatrix} a   & b \\ c   &   d \end{pmatrix} = \begin{pmatrix}
              a   &  \chi^{-e}(t)b \\
              \chi^{e}(t)c   &   d\end{pmatrix}$$
for all $t\in \TT$ and $\left(
\begin{smallmatrix}
a&b\\ c&d
\end{smallmatrix}
\right)\in \SL_2$.
The group $G_{e}$ is the semidirect product of $\SL_{2}$ and $\TT$ via the map $\varepsilon_{e}$. 
\end{rappel}
\begin{rappel}\label{Upm-Bpm}
In the sequel, we let $T\subset \SL_{2}$ be the maximal torus which consists
of diagonal matrices. We also consider the root subgroups 
$$ U_{-} = \left\{\left.\,\begin{pmatrix} 1   & 0 \\ \lambda   &  1  \end{pmatrix}\,\,\,\right|\,\,\,\lambda\in \KK\,\right\}
\,\,\,\,\,{\rm and} \,\,\,U_{+} = \left\{\left.\,\begin{pmatrix} 1   & \lambda \\ 0   &  1  \end{pmatrix} \,\,\,\right|\,\,\,
\lambda\in \KK\,\right\} $$
with respect to $T$.  We make also the convention that $\SL_2\subset G_e$ via the natural isomorphism $\SL_2\cong \SL_2\times\{1\}$.
We denote by $B_-\subset G_{e}$ (resp. $B_+\subset G_{e}$)  the Borel subgroup
generated by $U_-$ (resp. $U_+$) and $T\times\TT$.
\end{rappel}
\begin{rappel}
A choice of a $G_{e}$-action on an affine variety $Z$ 
is equivalent to a choice of (see \cite[Theorem 3]{CD})
\begin{itemize}
\item an $M$-grading on $\KK[Z]$
\item  and three derivations 
$\partial_{-},\delta, \partial_{+}$ on $\KK[Z]$
satisfying $$[\delta,\partial_\pm]=\pm2\partial_\pm \mbox{ and }$$ $$[\partial_+,\partial_-] = \partial_{+}\circ \partial_{-}
-\partial_{-}\circ \partial_{+}=\delta,$$
where $\partial_{\pm}$ is a homogeneous LND of degree $\pm e$ and $\delta$ 
is semisimple.   
\end{itemize} 
Note that $\partial_-$, $\partial_+$, and $\delta$ are given by the induced actions of $U_-$, $U_+$, and
 $T$ respectively, whereas the $M$-grading corresponds to the induced $\TT$-action.
Assume that the $\TT$-action on $Z$ is of complexity one. 
Then the variety $Z$ is $G_{e}$-spherical if and only if $\partial_{-}$ and $\partial_{+}$ are both of horizontal type. 
\end{rappel}
\begin{rappel}
By \cite[Corollary 2.3 (ii)]{AL}, in case $e=0$ we may consider the bigger torus and fall into $e\neq0$. So, the classification of affine $G_{e}$-spherical varieties splits into 
the following cases:
\begin{enumerate}
\item $e\neq 0$ and $Z$ is not toric for the action of a larger torus than $\mathbb{T}$. 
\item $e\neq0$ and $Z$ is toric for the $\TT$-action. 
\end{enumerate}
\end{rappel} 
\begin{definition}
An affine $G_{e}$-spherical variety $Z$ satisfying (1) (resp. (2)) is called of \emph{type} I (resp. of \emph{type} II).
\end{definition}
The following theorem gives a complete description of affine $G_{e}$-spherical varieties of type I
in terms of polyhedral divisors (see \cite[Theorem 2.18]{AL}). Here
$t$ is a local parameter on $\PP^{1}$ and $0,1,\infty\in\PP^{1}$
are chosen accordingly.
\begin{theorem}
\label{theorem:SL}
Let $X$ be an affine variety with a $G_{e}$-action such that the induced $\TT$-action is faithful of complexity one.
Then $X$ is $G_{e}$-spherical of type {\rm I} if and only if $\KK[X]$ is $G_{e}$-isomorphic to a rational 
$G_{e}$-algebra $A = A[C,\D]$ enjoying the following properties.
\begin{enumerate}[label=\emph{(\alph*)}]
\item $\D = \sum_{z\in C}\Delta_{z}\cdot z$ is a proper $\sigma$-polyhedral divisor over $C = \AA^{1}$ or $C = \PP^{1}$
supported in at most three points $0,1,\infty$.
\item The horizontal {\rm LND}s $\partial_{-},\partial_{+}$
are given by the coherent pairs
\footnote{By the abuse of notation we allow a coherent pair $(\dcol,e)$ to denote a triplet $(1,\dcol, e)$, cf. Theorem~\ref{theorem:hor}.} 
$(\D_{-},-e)$, $(\D_{+},e)$, where the marked point of $\D_\pm$ is $0$ and
 the point at the infinity of $\D_\pm$ is $\infty$.
\item $\D$, $\D_-$, and $\D_+$ are defined by the following conditions separated in two cases.
\begin{itemize}
 \item  
Either {\rm (}reflexive case{\rm )} $\Delta_0=\Conv(0,v_0)+\sigma$, $\Delta_1=\Conv(0,v_1)+\sigma$, where
$v_0,v_1\in N\setminus\sigma$, $v_0(e)=1$,  $v_1(e)=-1$, $\Delta_{\infty}(0)\subset e^{\perp}$, and\footnote{The
polyhedral cone $\tilde{\omega}_{\pm}$ is the domain of $\D_{\pm}$ (see \ref{marker:defcohe}).}
$(\pm e,-\frac{1}{2}\mp \frac{1}{2})\in\Rt(\tilde{\omega}_{\pm})$.
Representations of $\D_{-}, \D_{+}$ are given respectively by
\end{itemize}
$$\dots \lin \cou{v_{0}}{0} \lin \cou{0}{1} \lin \dots\,\,\,\,\,\,\,\,\,\dots \lin \cou{0}{0} \lin \cou{v_{1}}{1} \lin \dots$$
\begin{itemize}
\item Or {\rm (}skew case{\rm )} $\Delta_0=v_0+\sigma$, $\Delta_1=\Conv(0,v_1)+\sigma$, where
$2v_0,v_1\in N\setminus\sigma$, $2v_0(e)=1$,  $v_1(e)=-1$, $\Delta_{\infty}(0)\subset e^{\perp}$,
and $(\pm e,-\frac{1}{2}\mp \frac{1}{2})\in\Rt(\tilde{\omega}_{\pm})$.
Representations of $\D_{-}, \D_{+}$ are given respectively by
\end{itemize}
$$\dots \lin \cou{v_{0}}{0} \lin \cou{0}{1} \lin \dots\,\,\,\,\,\,\,\,\,\dots \lin \cou{v_{0}}{0} \lin \cou{v_{1}}{1} \lin \dots$$
\end{enumerate} 
The converse is also true: every choice of the combinatorial data $v_0,v_1,\sigma,\Delta_\infty$ satisfying the conditions above yields a $G_e$-spherical variety $X=\Spec A[C,\D]$ of type I.
\end{theorem}
\begin{rappel}\label{rap:SL2-omegas}
The resulting cones $\omega_\pm$ in Theorem~\ref{theorem:SL} are described as follows. 
\begin{center}
\begin{tabular}{c|c}
reflexive & skew \\ \hline
$\omega_-=\Cone(-v_0,v_1)+\sigma$ & $\omega_-=\Cone(v_1)+\sigma$ \\
$\omega_+=\Cone(v_0,-v_1)+\sigma$&$\omega_+=\Cone(-v_1)+\sigma$\\
\end{tabular}
\end{center}
The cones $\tilde\omega_\pm$ depend on the curve $C$, we denote them by $\tilde\omega_\pm^{\AA^{1}}$ and $\tilde\omega_\pm^{\PP^{1}}$
for $C=\AA^{1}$ and $C=\PP^{1}$ respectively. Then

\begin{center}
\begin{tabular}{c|c}
reflexive & skew \\ \hline
$\tilde\omega_-^{\AA^{1}}=(\omega_-,0)+\Cone((v_0,1))$& $\tilde\omega_-^{\AA^{1}}=(\omega_-,0)+\Cone((v_0,1))$\\
$\tilde\omega_+^{\AA^{1}}=(\omega_+,0)+\Cone((0,1))$& $\tilde\omega_+^{\AA^{1}}=(\omega_+,0)+\Cone((v_0,1))$\\
$\tilde\omega_-^{\PP^{1}}=\tilde\omega_-^{\AA^{1}}+\Cone((\Delta_\infty,-1))$& $\tilde\omega_-^{\PP^{1}}=\tilde\omega_-^{\AA^{1}}+\Cone((\Delta_\infty,-1))$\\
$\tilde\omega_+^{\PP^{1}}=\tilde\omega_+^{\AA^{1}}+\Cone((\Delta_\infty+v_1,-1))$& $\tilde\omega_+^{\PP^{1}}=\tilde\omega_+^{\AA^{1}}+\Cone((\Delta_\infty,-1))$\\
\end{tabular}
\end{center}
The cones $\tilde\omega_\pm^{\PP^{1}}$ and $\tilde\omega_\pm^{\AA^{1}}$ coincide in the non-negative half-space, i.e. 
\[\tilde\omega_\pm^{\PP^{1}}\cap \{(v,p)\;|\; v\in N_\QQ, p\in\QQ_{\ge0}\}=\tilde\omega_\pm^{\AA^{1}}.\]
\end{rappel}

\begin{rappel}
\label{rap:remdpm}
The {\rm LND} $\partial_\pm$ corresponding to the coherent pair $(\D_{\pm}, \pm e)$ in Theorem \ref{theorem:SL} is 
defined via the formulae (see the comment before \cite[Lemma 2.16]{AL} and the proof of \cite[Lemma 2.17]{AL})

either (reflexive case)
\begin{align*}
\partial_{\pm}(t^{r}\xi^{\pm}_{m}\chi^{m})=\left[\left(\frac{1}{2}\mp\frac{1}{2}\right)v_0(m)+r\right]\cdot t^{r- \frac{1}{2}\mp\frac{1}{2}}\xi^{\pm}_{m\pm e}\chi^{m\pm e}\,\,\,\,{\rm for\,\,all}\,\,\,(m,r)\in M\oplus\mathbb{Z},
 \\
\delta(t^{r}\chi^{m}) = (v_{0} - v_{1})(m)\cdot t^{r}\chi^{m}\,\,\,\,{\rm for\,\,all}\,\,\,(m,r)\in M\oplus\mathbb{Z},
\end{align*}
where $\xi^{\pm}_m=(t-1)^{-(\frac{1}{2}\pm \frac{1}{2})v_1(m)};$

or (skew case)
\begin{align*}
\partial_{\pm}(t^{r}\xi_{m}^{\pm}\chi^{m})=2(v_0(m)+r)\cdot t^{r - \frac{1}{2}\mp\frac{1}{2}}\xi^{\pm}_{m\pm e}\chi^{m\pm e}\,\,\,\,{\rm for\,\,all}\,\,\,(m,r)\in M\oplus\mathbb{Z},
 \\
\delta(t^{r}\chi^{m}) =  - 2v_{1}(m)\cdot t^{r}\chi^{m}\,\,\,\,{\rm for\,\,all}\,\,\,(m,r)\in M\oplus\mathbb{Z},
\end{align*}
where $\xi^{\pm}_m$ is again $(t-1)^{-(\frac{1}{2}\pm \frac{1}{2})v_1(m)}.$ 
\end{rappel}

\begin{rappel}\label{rap:der-Q}
The formulae for $\partial_\pm$ can also be expressed for an arbitrary $Q\in\KK(t)$.
In reflexive case they are:
\begin{align*}
\partial_-(Q\chi^{m})=&(v_0(m)Q+tQ')\chi^{m-e},\\
\partial_+(Q\chi^{m})=&(v_1(m)Q+(t-1)Q')\chi^{m+e},
\end{align*}
where $Q'=\frac{dQ}{dt}$. In skew case they are:
\begin{align*}
\partial_-(Q\chi^{m})=&2(v_0(m)Q+tQ')\chi^{m-e},\\
\partial_+(Q\chi^{m})=&2\left(\left(v_0(m)\left(1-\frac{1}{t}\right)+v_1(m)\right)Q+(t-1)Q'\right)\chi^{m+e}.
\end{align*}
\end{rappel}


In fact, every skew $G_e$-spherical variety of type I is obtained from a reflexive one as the quotient by an involution commuting with the $G_e$-action.
\begin{proposition}
Let $X=\Spec A[C,\D]$ be a skew $G_e$-spherical variety as in Theorem~\ref{theorem:SL}, and $\tilde X=\Spec A[C,\tilde\D]$ be the reflexive $G_e$-spherical variety defined by the pair $\tilde v_0,\tilde v_1$ such that $\tilde v_1=v_1$ and $\tilde v_0=-v_1$, $\tilde\sigma=\sigma$, and if $C=\PP^1$, $\tilde\Delta_\infty=2\Delta_\infty+2v_0+v_1$. 
Introduce the parity function $\pa\colon M\to\{0,1\}$ via $\pa(m)=0$ if $v_0(m)\in \ZZ$ and $\pa(m)=1$ otherwise. Consider the following involution on $\K[\tilde X]$:
\[\tau\colon t\mapsto 1-t,\; \chi^m\mapsto (-1)^{\pa(m)}\left(\frac{1-t}{t}\right)^{v_1(m)}\chi^m\; \mbox{ for any }\; m\in M,\]
Then $\tau$ commutes with the $G_e$-action, and $X$ is the quotient of $\tilde X$ by this involution.
\end{proposition}
\begin{proof}

For $u\in\KK[t]$ and $d\in\QQ$ we denote by $\K[u]_{\le d}$
the subspace of $\KK[u]$ generated by elements of degree $\le d$ in $u$.
We assume that $C=\PP^1$; the case $C=\AA^1$ is similar and left to reader.
The relation $\tilde\Delta_\infty=\Delta_\infty+2v_0+v_1$ ensures that $\deg\tilde\D=2\deg\D$, so $\tilde\D$ is proper.
Then by construction
$A[C,\tilde\D]=\bigoplus_{m\in\sigma^\vee\cap M} \tilde A_m\chi^m$, where 
\[\tilde A_m=H_0(C, \Oc_C(\lfloor\D(m)\rfloor))=\begin{cases} \K[t]_{\le\deg\tilde\D(m)}t^{v_1(m)}&\mbox{ if } v_1(m)\ge0, \\\K[t]_{\le\deg\tilde\D(m)}(t-1)^{-v_1(m)}&\mbox{ if }v_1(m)<0. \end{cases}\] 
It follows that $A[ C,\tilde\D]^\tau=\bigoplus_{m\in\sigma^\vee\cap M} (\tilde A_m\chi^m)^\tau,$ where
\[(\tilde A_m\chi^m)^\tau=\begin{cases} \K[t(t-1)]_{\le\deg\D(m)}(2t-1)^{\pa(m)}t^{v_1(m)}\chi^m&\mbox{ if } v_1(m)\ge0, \\\K[t(t-1)]_{\le\deg\D(m)}(2t-1)^{\pa(m)}(t-1)^{-v_1(m)}\chi^m&\mbox{ if }v_1(m)<0. \end{cases}\]
On the other hand, $A[C,\D]=\bigoplus_{m\in\sigma^\vee\cap M} A_m\chi^m,$ where
\[A_m=\begin{cases} \K[t]_{\le\deg\D(m)}t^{-v_0(m)+\pa(m)/2}&\mbox{ if } v_1(m)\ge0, \\\K[t]_{\le\deg\D(m)}t^{-v_0(m)+\pa(m)/2}(t-1)^{-v_1(m)}&\mbox{ if }v_1(m)<0. \end{cases}\] 
The $G_e$-equivariant isomorphism $\phi\colon A[C,\D]\to A[C,\tilde\D]^\tau$ is given by
\[
\phi(t)=(2t-1)^2,\, \phi(\chi^m)=t^{v_1(m)}\left(t-\frac{1}{2}\right)^{2v_0(m)}\chi^{m}.
\]
By direct computation we verify that $\partial_\pm(\phi(t))=\phi(\partial_\pm(t))$ and $\partial_\pm(\phi(\chi^m))=\phi(\partial_\pm(\chi^m))$ for any $m\in\sigma^\vee\cap M$. Note that one the left hand side we have reflexive derivations and on the right hand side we have skew ones. Thus, $\phi$ is $G_e$-equivariant.
\end{proof}

\section{Colored cones versus colored polyhedral divisors}\label{sec:3}
Let $G_{e}$ be the previously introduced reductive group $\SL_{2}\rtimes_{e}\TT$.
The purpose of this section is to express the approach of Arzhantsev--Liendo
 in terms of the Luna--Vust theory applied to affine $G_{e}$-spherical varieties. 
Sections \ref{subsec:3.1} and \ref{subsec:3.2} are devoted to the type I and Section \ref{subsec:3.3} to the type II.

\subsection{Classification of type I}
\

\label{subsec:3.1}
Let $X$ be an affine $G_{e}$-spherical variety 
of type I. Suppose that $\KK[X] = A[C,\D]$
admits a presentation as in \ref{theorem:SL} (either skew or reflexive case).
We identify the open $G_e$-orbit of $X$ with a homogeneous space $G_e\,/H$ so that $X$ is its embedding.

\begin{remark}\label{GH-aff}
Any non-trivial unipotent subgroup of $G_e$ is one-dimensional, hence maximal. 
 Since $X$ is not horospherical \cite[Proposition 3.8]{AL}, $H$ does not contain a unipotent subgroup, thus it is reductive.
So, by the Matsushima criterion, $G_{e}\,/H$ is affine.
\end{remark}
\begin{rappel}
\label{marker:not}
Denote by   
$$\dots \lin \cou{v_{0}^{-}}{0} \lin \cou{v_{1}^{-}}{1} \lin \dots\,\,\,\,\,\,\,\,\,\dots \lin \cou{v_{0}^{+}}{0} \lin 
\cou{v_{1}^{+}}{1} \lin \dots$$
the representations of $\D_{-}$ and $\D_{+}$ respectively. For instance, in the reflexive case,
$v_{0}^{-} = v_{0}$, $v_{1}^{-} = 0$, $v_{0}^{+} = 0$, and $v_{1}^{+} = v_{1}$. 
For every $m\in M$ we let
\begin{eqnarray*}
\phi_{m}^{\pm} = (t-1)^{-v_{1}^{\pm}(m)}t^{-\lfloor v_{0}^{\pm}(m)\rfloor}. 
\end{eqnarray*}
Then we have
\begin{equation}\label{eq:kerpm}
\ker\,\partial_\pm=\bigoplus_{m\in\omega^{\vee}_{\pm}\cap L}\KK\phi^{\pm}_{m}\chi^{m},
\end{equation}
where $\omega_{\pm}$ is the associated cone of $\D_{\pm}$ and 
$$L=\{m\in M\;|\;v_0^{+}(m)\in\ZZ\}=\{m\in M\;|\;v_0^{-}(m)\in\ZZ\}.$$
\end{rappel}
In the following lemma we use the notation from \ref{Upm-Bpm}.
\begin{lemma}
\label{lemma:B-orbit}
Let $A^{\pm}_M=\bigoplus_{m\in M}\KK[t]\cdot\phi^{\pm}_{m}\chi^{m}$.
Then the following assertions hold.
\begin{enumerate}[label=\emph{(\roman*)}]
\item The lattice of $B_{\pm}$-weights of the $B_{\pm}$-algebra $\KK(X)$ is $L$.
\item The variety $X_M^{\pm}=\Spec A_M^{\pm}$ is a principal open subset of $X$ 
which coincides with the open $B_{\pm}$-orbit.
\end{enumerate}
More precisely, the supports of all divisors $\div(\phi^{\pm}_m\chi^{m})$ for $m\in \relint(\omega^{\vee}_\pm)\cap L$ 
are equal and form the complement  $X\setminus X_{M}^{\pm}$.
\end{lemma}

\begin{proof}
(i) The first assertion follows from the fact that $B_{\pm}$-eigenvectors in $\KK(X)$ are exactly quotients of non-zero homogeneous elements of $\ker\,\partial_\pm$. 

(ii)
By Remark~\ref{omega-sens}, we can add a principal polyhedral divisor to $\D$ so that for any $m\in  \relint(\omega^{\vee}_\pm)\cap L$ the divisor $\D(m)|_{\AA^{1}}$ is supported in at most 
the point $0$ and that $\phi_{m}^{\pm} = t^{-\lfloor v_{0}^{\pm}(m)\rfloor}$. 
Let us fix $m\in  \relint(\omega^{\vee}_\pm)\cap L$ and denote $\alpha_\pm = \phi_m^{\pm}\chi^{m}$.
By \cite[Proposition 3.3]{AHS}, the localization 
$\KK[X]_{\alpha_{\pm}}$ is equal to $A_M^{\pm}$. 
Since $\alpha_{\pm}\in {\rm ker}\,\partial_{\pm}$, the derivation $\partial_{\pm}$
extends to an LND on $A^{\pm}_{M}$ written by the same letter $\partial_{\pm}$.
Hence, $X^{\pm}_{M}$ is a $B_{\pm}$-stable principal open subset of $X$.

Let us deduce that $X^{\pm}_M$ is the open $B_\pm$-orbit. 
Denote by  $E_M^{\pm}$ the normalization of the algebra $A_M^{\pm}[\sqrt[d]{\phi_{\pm d e}^{\pm}}\chi^{\pm e}]$,
where $d = d(e)$. 
Then the inclusion $A_M^{\pm}\subset E_M^{\pm}$ yields a finite surjective map 
$$\pi\colon Z_M^{\pm}=\Spec  E_M^{\pm} {\to} X_M^{\pm}.$$ 
By \cite[Corollary 2.6]{LL}, the derivation $\partial_\pm$ admits a unique extension to $E_M^{\pm}$, hence there is a natural $B_{\pm}$-action on $Z^{\pm}_{M}$ and $\pi$ is $B_{\pm}$-equivariant. 
Denote  $\zeta=\sqrt[d]{t}$. The map 
$$\gamma\colon E_M^{\pm}\to\KK[M][\zeta]=\bigoplus_{m\in M}\KK[\zeta]\chi^{m},\quad \zeta^{l}\chi^{m}\mapsto\zeta^{l+dv_0^{\pm}(m)}\chi^{m}$$
is an isomorphism of $M$-graded algebras.
Consider the extension of $\partial_\pm$ to $E_M^{\pm}$ denoted by  $\tilde\partial_\pm$ and introduce
 $\tilde\partial_\pm^{\prime}=\gamma\circ\tilde\partial_\pm\circ\gamma^{-1}$,  which is a homogeneous LND on 
$\KK[M][\zeta]$. Then the induced $B_\pm$-action on  $\Spec\KK[M][\zeta]\cong\TT\times\AA^{1}$ is transitive, 
and so is the $B_\pm$-action on $Z_M^{\pm}$. Since $\pi$ is $B_\pm$-equivariant 
and surjective, the $B_\pm$-action on $X_M^{\pm}$ is also transitive.
\end{proof}
Now we give an explicit description of the $B_{\pm}$-divisors on $X$ and
compute their images in the vector space $V_{\QQ}$ (see \ref{marker:val1}, \ref{marker:val2}),
where $V = {\rm Hom}(L,\mathbb{Z})$ is the valuation lattice.
Note that $V_{\QQ}$ is identified with $N_{\QQ}$.
\begin{proposition}
\label{proposition:col.G-div}
Denote by $\vrho_{\pm}$ the map that sends a discrete valuation $\nu$ of $\KK(X)$
into the linear form $\vrho_{\pm}(\nu)\in V_{\QQ}$ defined as follows. For any $B_{\pm}$-eigenvector $\alpha\in\KK(X)$ of weight $\chi_{\alpha}\in L$ 
we let $$\vrho_{\pm}(\nu)(\chi_{\alpha}) = \nu(\alpha).$$
To simplify the notation, we also let $\vrho_\pm(D) = \vrho_\pm(v_{D})$ for a prime divisor $D\subset X$.
Then the colors of $X$ are given in the following table in accordance with notation \ref{rappel:div-hor-vert}.
\
\newline

\begin{tabular}{|l|l|l|l|}
\hline
&  Reflexive case   &  Skew case  \\
   \hline
    $B_{-}$-action & $\vrho_{-}(D_{(0,0)}) = -v_{0}$, $\vrho_{-}(D_{(1, v_{1})}) = v_{1}$ &   $\vrho_{-}(D_{(1,v_{1})}) = v_{1}$ \\
   \hline
    $B_{+}$-action  &  $\vrho_{+}(D_{(0,v_{0})}) = v_{0}$, $\vrho_{+}(D_{(1, 0)}) = -v_{1}$& $\vrho_{+}(D_{(1,0)}) = -v_{1}$ \\
   \hline
\end{tabular}
\
\newline

Moreover, the $G_{e}$-divisors on $X$ are exhausted by $D_\rho$ for $\rho\in\sistar$ and
 $D_{(\infty,v)}$ for $v\in\Delta_\infty(0)$. We have $\vrho_{\pm}(D_{\rho}) = \rho$
for any $\rho\in\sistar$. The images of the vertical $G_{e}$-divisors 
are given in the following table.
\
\newline
%

\begin{tabular}{|l|l|l|l|}
\hline
& Reflexive case  &  Skew case   \\
   \hline
     $B_{-}$-action  & $\vrho_{-}(D_{(\infty, v)}) = \mu(v)(v_{0} + v)$ &  $\vrho_{-}(D_{(\infty,v)}) = \mu(v) (v_{0} + v)$ \\
   \hline
    $B_{+}$-action & $\vrho_{+}(D_{(\infty, v)}) = \mu(v)(v_{1} + v)$ &  $\vrho_{+}(D_{(\infty,v)}) = \mu(v) (v_{0} + v_{1} + v)$ \\
   \hline
\end{tabular}
\
\newline

\end{proposition}

\begin{proof}
We will illustrate our calculation method on the $B_{+}$-action in the reflexive case.
It can be easily adapted to the remaining cases that are left to the reader.

A representation of $\D_{+}$ in the reflexive case is  given by
$$\dots \lin \cou{0}{0} \lin \cou{v_{1}}{1} \lin \dots$$
Let $m\in \relint(\omega^{\vee}_{+})\cap L$. We have $\phi^{+}_{m}=(t-1)^{-v_1(m)}$.
By Lemma \ref{lemma:B-orbit}, the $B_{+}$-divisors of $X$ are exactly the irreducible components of 
the support of the principal divisor 
$$\div(\phi^{+}_{m}\chi^{m})=\sum_{\rho\in\sistar} \rho(m)\cdot D_\rho
+ \sum_{z\in\{0,1,\infty\}}\sum_{v\in\Delta_z(0)}\mu(v)(v(m) +{\rm ord}_{z}\,\varphi^{+}_{m})\cdot D_{(z,v)}.$$
Since $m\in \relint(\sigma^{\vee})$, the integer $\rho(m)$ is
non-zero for every $\rho\in\sistar$. Hence all the horizontal divisors on $X$
are $B_{+}$-stable. 

An easy computation shows that the vertical part of $\div(\phi^{+}_{m}\chi^{m})$ is
$$\div(\phi^{+}_{m}\chi^{m})_{{\rm vert}} = \Sigma_{0} +\Sigma_{1} +\Sigma_{\infty},$$ 
where 
\begin{align*}
\Sigma_0= v_0(m)\cdot D_{(0,v_0)},\\
\Sigma_1=-v_1(m)\cdot D_{(1,0)},\\
\Sigma_\infty= \sum_{v\in\Delta_{\infty}(0)}\mu(v) (v+v_1)(m)\cdot D_{(\infty,v)}.
\end{align*}
We may choose $m$ such that $v_{0}(m)$ and $v_{1}(m)$ are nonzero. By Lemma \ref{lemma:B-orbit},
this holds for every $m\in \relint(\omega^{\vee}_{+})\cap L$. Hence 
the divisors $\Sigma_{0}$ and $\Sigma_{1}$
are nonzero, and so $D_{(0,v_{0})}, D_{(1,0)}$ are $B_{+}$-stable. Furthermore,
since $\D$ is proper,
\begin{eqnarray*}
{\rm ord}_{\infty}\,\phi^{+}_{m} + \min_{v\in\Delta_{\infty}(0)}v(m)>0. 
\end{eqnarray*}
Thus, for any $v\in\Delta_{\infty}(0)$, $(v_{1} + v)(m)>0$. This shows that
the vertical divisors of the form $D_{(\infty, v)}$ are $B_{+}$-stable. 

Making the same computation for the divisor ${\rm div}(\phi^{-}_{m}\chi^{m})$ where $m\in \relint(\omega^{\vee}_{-})\cap L$,
we obtain that the divisors of the form $D_{\rho}$ and $D_{(\infty, v)}$ are $B_{-}$-stable.
Hence they form the set of $G_{e}$-divisors of $X$.

The divisors $D_{(0,v_{0})}$ and $D_{(1,0)}$ do not 
appear in the support of ${\rm div}(\phi^{-}_{m}\chi^{m})$. This implies that
$D_{(0,v_{0})}$, $D_{(1, 0)}$ are the colors of $X$ for the $B_{+}$-action.
Finally, one concludes by the formula
$$\div(\phi^{+}_{m}\chi^{m})=\sum_{\rho\in\sistar} \vrho_{+}(D_{\rho})(m)\cdot D_\rho
+ \sum_{z\in\{0,1,\infty\}}\sum_{v\in\Delta_z(0)}\vrho_{+}(D_{(z,v)})(m) \cdot D_{(z,v)}.$$
\end{proof}

\begin{remarque}
\label{marker:remarkG/H}
Let us consider the subalgebra $A^{h}=A[\AA^{1},\D_{G_{e}\,/H}]\subset\KK(t)[M]$ defined by the polyhedral divisor 
$$\D_{G_{e}\,/H} = \Delta(G_{e}\,/H)_{0}\cdot \{0\}	
+ \Delta(G_{e}\,/H)_{1}\cdot \{1\}$$
with the conditions $\Delta(G_{e}\,/H)_{0} = {\rm Conv}(0,v_{0})$, $\Delta(G_{e}\,/H)_{1} ={\rm Conv}(0, v_{1})$ 
in the reflexive case, and  $\Delta(G_{e}\,/H)_{0} = \{v_{0}\}$, $\Delta(G_{e}\,/H)_{1} = {\rm Conv}(0, v_{1})$
otherwise.

Then $\KK[X]\subset A^{h}$, and $A^{h}$ is $G_e$-stable as a subalgebra of $\KK(X)$. 
Therefore, $X^{h}=\Spec A^{h}$ is an embedding of $G_e\,/H$, see \ref{marker:emb}.
Since $X^{h}$ does not include $G_e$-divisors by Proposition~\ref{proposition:col.G-div}, and $G_e\,/H$ is affine by Remark~\ref{GH-aff},  $X^{h}=G_e\,/H$.
So, the polyhedral divisor $\D_{G_e\,/H}$ characterizes the open orbit. Thus, the subgroup $H$ determines whether $X$ is skew, reflexive or of type II.
\end{remarque}

 The closed $G_{e}$-orbit $Y\subset X$ is described as follows.
\begin{lemme}
\label{lemma:closorb}
The vanishing ideal of $Y$ is
\begin{eqnarray*}
I=\bigoplus_{m\in(\sigma^{\vee}\setminus\lins)\cap M} A_{m}\chi^{m},\,\,\,
{\it where}\,\,\,A_{m} = H^{0}(C,\mathcal{O}_{C}(\lfloor \D(m)\rfloor)).
\end{eqnarray*} 
\end{lemme}
\begin{proof}
Let $R=\KK[X]/I$ and $Z = {\rm Spec}\, R$. We will prove that $Z=Y$.

\emph{Case $C = \AA^{1}$}. Since $I$ is generated by 
$\bigcup_{\rho\in\sigma(1)}I(D_{\rho})$,
the ideal $I$ is $G_{e}$-stable and $R$ is  a rational $G_{e}$-algebra. 
Let us describe $R$ in terms of polyhedral divisors.
Consider the projection $\pi :N\rightarrow \tilde{N}$, where $\tilde{N} = N/(\lins^{\perp}\cap N)$.
The lattice $\tilde{N}$ is the dual of $\lins\cap M$.
Let $\tilde{\D}$ be the polyhedral divisor with coefficients in $\tilde{N}_{\mathbb{Q}} = \mathbb{Q}\otimes_{\mathbb{Z}}\tilde{N}$ 
given by
\begin{eqnarray*}
\tilde{\D} = {\rm Conv}(0, \pi(v_{0}))\cdot \{0\} + {\rm Conv}(0, \pi(v_{1}))\cdot \{1\},
\end{eqnarray*}
in the reflexive case, and by
\begin{eqnarray*}
\tilde{\D} = \{\pi(v_{0})\}\cdot \{0\} + {\rm Conv}(0, \pi(v_{1}))\cdot \{1\}
\end{eqnarray*}
in the skew case. Note that $e\in \lins\setminus \{0\}$ (see \cite[Lemma 3.1]{Li}) and $\pi(v_{0}),\pi(v_{1})$ are nonzero. 
We have a $G_e$-equivariant isomorphism $R \cong A[C,\tilde{\D}]$. The induced homogeneous LNDs on $A[C,\tilde{\D}]$ are horizontal and represented by the graphs
$$\dots \lin \cou{\pi(v_{0}^{-})}{0} \lin \cou{\,\,\,\,\,\,\pi(v_{1}^{-})}{1} \lin \dots\,\,\,\,\,\,\,\,\,\dots \lin \cou{\pi(v_{0}^{+})}{0} \lin \cou{\,\,\,\,\,\,\pi(v_{1}^{+})}{1} 
\lin \dots$$ In particular, $Z$ is an affine spherical
variety of type I which is homogeneous under $G_{e}$,
according to Remark \ref{marker:remarkG/H}. So, $Z$ is the closed $G_e$-orbit.

\emph{Case $C = \PP^{1}$}. By properness of $\D$, every homogeneous element of degree $m\in\lins\cap M$
is invertible and so belongs to ${\rm ker}\,\partial_{-}\cap {\rm ker}\,\partial_{+}$.
Hence $I$ is $G_{e}$-stable. Furthermore, the $\TT$-action on $Z$ is transitive. This allows us to conclude.   
\end{proof}

The following theorem provides the Luna--Vust description of the affine $G_{e}$-spherical variety $X$
in terms of its AL-colorings $\D_{-},\D_{+}$. 
\begin{theorem}
\label{theorem:first}
Let $X$ be an affine $G_{e}$-spherical variety admitting a presentation as in 
Theorem~\ref{theorem:SL}.
Then the following statements hold.
\begin{enumerate}
\item[(i)] If $C = \AA^{1}$ then $X$ is a toroidal spherical variety with colored cone $(\sigma,\emptyset)$. 
\item[(ii)] If $C = \PP^{1}$ then the colored cone of $X$ is given by the following table (cf.~\ref{rap:SL2-omegas}).  
\
\newline


\begin{tabular}{|l|l|l|l|}
\hline
&  Reflexive case  & Skew case   \\
   \hline
     $B_{-}$-action & $\left(\omega_{-}, \{D_{(0,0)}, D_{(1, v_{1})}\}\right)$ &$\left(\omega_{-}, D_{(1, v_{1})}\right)$  \\
   \hline
    $B_{+}$-action  &  $\left(\omega_{+}, \{D_{(0, v_{0})}, D_{(1,0)}\}\right)$ & $\left(\omega_{+}, D_{(1,0)}\right)$ \\
   \hline
\end{tabular}
\end{enumerate}
 \
\newline
\end{theorem}

\begin{proof}
By Proposition \ref{proposition:col.G-div} 
and Lemma \ref{lemma:closorb}, the closed orbit is contained in all colors of $X$ in the case $C = \PP^{1}$ and 
is not contained in any colors in the case
$C = \AA^{1}$. In the former case, one concludes 
by Lemma \ref{lemma:affcone}.  
\end{proof}
\subsection{Horizontal and vertical $G_e$-invariant valuations}
\

\label{subsec:3.2}
As before, $X$ is an affine embedding of type I of a spherical homogeneous space $G_{e}\,/H$
with the associated colored cone $(\Cc,\F_{Y})$. In this subsection, we provide the representation of $\KK[X]$ 
by a polyhedral divisor and AL-colorings in terms
of $(\Cc,\F_{Y})$. For this, we examine the valuation cone of $G_{e}\,/H$. 
\begin{rappel}
Without loss of generality, we may suppose that $\KK[G_{e}\,/H]$ is equal to the $G_{e}$-algebra 
$A[\AA^{1},\D_{G_{e}\,/H}]$ defined in Remark~\ref{marker:remarkG/H}.
Thus, a variety $Z$ is an affine embedding of $G_{e}\,/H$ if and only if 
$\KK[Z]$ is a $G_{e}$-stable normal affine subalgebra of $A[\AA^{1},\D_{G_{e}\,/H}]$. 
\end{rappel}
Inspired by \cite{Ti2}, we introduce the following
terminology.
\begin{definition}
A non-trivial $G_{e}$-invariant valuation $v$ on $\KK(G_{e}\,/H)$ is called \emph{horizontal} (or \emph{central}) if $v$ is trivial on
the subfield $\KK(t)$. Otherwise, $v$ is said to be \emph{vertical}. We denote by
$\vhor$ (resp. $\vvert$) the subset of horizontal (resp. vertical)
valuations in the valuation cone $\Vc$ of $G_{e}\,/H$. 
\end{definition}
 We make the convention that $\vhor\cap \vvert = \{0\}$, i.e. the trivial valuation is both horizontal
and vertical. 
For the next result, we  emphasize that vectors in $N_\QQ$ that belong to the valuation lattice $V = {\rm Hom}(L,\mathbb{Z})$ are regarded as integral. 
\begin{lemma}
\label{lemma:valhorvert}
 Consider the $B_\pm$-action on $G_e\,/H$.
 Then $\Vc=(\QQ_{\ge0}(\mp e))^{\vee}$.
Further, denote $$S=\Cone(-\varrho_\pm(\F_0)),$$
where $\F_0$ is the set of colors of $G_e\,/H$.
Then 
$\vhor\cap(\Vc\setminus S)=e^{\perp}\setminus\{0\}$.
\end{lemma}
\begin{figure}[h]
\begin{center} 
\[  \begin{tikzpicture}
  [scale=.57, vertex/.style={circle,draw=black!100,fill=black!100,thick, inner sep=0.5pt,minimum size=0.5mm}, cone/.style={very thick,>=stealth}]
  \filldraw[fill=black!20!white,draw=white!100]
    (0,4.2) -- (4.2,4.2) -- (4.2,-4.2) -- (0,-4.2) -- cycle;
  \draw[cone] (0,-5.4) -- (0,5.4);
  \node at (5.3,0) {$\Vc$};
  \draw[cone] (0,0) -- (-3,1);
  \node at (-2.7,1.2) {$v_0$};
  \draw[cone] (0,0) -- (-3,-2);
  \node at (-3.4,-2.2) {$-v_1$};
  \draw[cone] (0,0) -- (6,-2);
  \draw[cone] (0,0) -- (6,4);
  \filldraw[fill=white,draw=white!100]
    (6,-2) -- (0,0) -- (6,4)  -- cycle;
  \node at (5,1) {$S$};
  \node at (.5,-5) {$e^{\perp}$};
  \foreach \x in {-4,-3,...,4}
  \foreach \y in {-4,-3,...,4}
  {
    \node[vertex] at (\x,\y) {};
  }
  \end{tikzpicture}\]
\end{center} 
\caption{The gray area corresponds to $\Vc\setminus S$ in $N_\QQ$ for $B_+$-action.}
\end{figure}
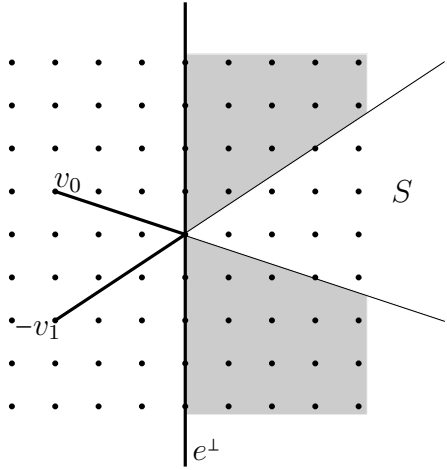

\begin{proof}
Let us show that $\Vc=\Cone(\mp e)^{\vee}$ with respect to the $B_\pm$-action.
For any  primitive $\rho\in e^{\perp}\setminus\{0\}$
, let us define the affine embedding $Z=\Spec A[\AA^{1},\D_\rho]$ of $G_e\,/H$ via
\begin{align*}
\D_\rho=\Delta_0^{\rho}\cdot\{0\}+\Delta_1^{\rho}\cdot\{1\},
\end{align*}
where $\Delta_0^{\rho}$ and $\Delta_1^{\rho}$ are as in Theorem~\ref{theorem:SL} with $\sigma=\Cone(\rho).$
Then by Proposition~\ref{proposition:col.G-div}, $Z$ has a unique horizontal $G_e$-divisor $D_\rho$, hence $\rho=\varrho_\pm(D_\rho)$ is contained in $\Vc$. This yields $e^{\perp}\subset\vhor$.
Using again Proposition~\ref{proposition:col.G-div}, we may construct an affine embedding of $G_e\,/H$ having a vertical $G_e$-divisor $D_{(v,\infty)}$ such that 
$\varrho_\pm(D_{(v,\infty)})\in\{v\in N_\QQ\;|\; \mp e(v)>0\}.$ Hence
$$\Cone((\mp e)^{\vee})=\Cone\left(e^{\perp}\cup\varrho_\pm(D_{(v,\infty)})\right)\subset \Vc.$$ 
Since $G_e\,/H$ is not horospherical, by \cite[Corollary 6.2]{Kn} we obtain  $\Vc= \Cone(\mp e)^{\vee}.$

Take a primitive vector $v\in \Vc\setminus S$ and consider the colored cone $(\Cc,\F_Y)$ of  $G_e\,/H$ given by the following conditions:
\begin{enumerate}
\item $\F_Y$ is the set of all colors of $G_e\,/H$;
\item $\Cc=\QQ_{\ge0}v+\sum_{D\in\F_Y}\QQ_{\ge0}\varrho(D)$.
\end{enumerate}
We note that $\Cc$ is strongly convex, since $v\notin S$.
Assume further that $v\notin e^{\perp}$. In this case $(\Cc,\F_Y)$ is an affine colored cone. Indeed,  $\varrho(\F_Y)\cap\Vc=\emptyset$ by Proposition~\ref{proposition:col.G-div}, and we set $m=0$ in Definition~\ref{definition:conecol}.
Let $Z'$ be the embedding corresponding to $(\Cc,\F_Y)$. Then $Z'$ is not toroidal and so admits a presentation $\KK[Z']=A[\PP^{1},\D']$
as in Theorem~\ref{theorem:SL}. By \ref{proposition:col.G-div}, the only $G_e$-divisor $D'$ on $Z'$ is vertical and $\vrho(D')=v$.
Thus, $v\in\vvert$. The assertion follows.
\end{proof}
\begin{remark}
The primitive vectors in $\Vc\setminus S$ are exactly the images of the $G_e$-divisors of all affine embeddings of $G_e\,/H$.
\end{remark}
\begin{definition}
Given a vector $v\in N_\QQ\setminus e^{\perp}$, we introduce the following function $\beta_v\colon N_\QQ\to \QQ.$ Consider the linear projection
$\pi_v\colon N_\QQ\to \QQ\cdot v$ along $e^{\perp}$. Then given $x\in N_\QQ$, $\beta_v(x)$ is the unique rational number such that  $\pi_v(x)= \beta_v(x)v$.  We say that $\pi_v$ is the projection defining $\beta_v.$
\end{definition}

This notation is used in the following proposition that introduces the explicit construction of the AL-colorings corresponding to an affine embedding of $G_e\,/H$ of type I. This is the inverse of Theorem~\ref{theorem:first}.
\begin{proposition}
\label{propdivpolversuscol}
Let $X$ be an affine $G_e$-spherical variety of type I, whose open $G_e$-orbit is identified with the homogeneous space $G_e\,/H$.
Let $\beta_-$ be $\beta_{v_0}$ and $\beta_+$ be either $\beta_{v_1}$ in the reflexive case or $\beta_{v_0+v_1}$ in the skew case (cf. latter table in Proposition~\ref{proposition:col.G-div}).
Denote
$$ \cvert_\pm=
\left\{\left.
\frac{v}{\beta_\pm (v)}
\;
\right|
\;
v\in \Cc(1),\, v(e)>0
\right\}.$$
Let $\sigma\subset N_{\QQ}$ be the strongly convex polyhedral cone 
equal to either
$$\sigma = \Cone(\Cc(1)\cap e^{\perp},\cvert_\pm,\cvert_\pm\pm v_0,\cvert_\pm\mp v_1,\cvert_\pm\pm v_0\mp v_1)$$
in the reflexive case or
$$\sigma = \Cone(\Cc(1)\cap e^{\perp},\cvert_\pm,\cvert_\pm\mp v_1)$$
in the skew case.
We have $\KK[X] = A[C,\D]$,
where $\D$ is a $\sigma$-polyhedral divisor over
$C = \AA^{1}$, if $\Cc\subset e^{\perp}$, and over $C = \PP^{1}$, otherwise. The non-trivial coefficients of $\D$ 
are defined as follows 
{\rm (}see \ref{marker:remarkG/H}{\rm )}. 
$$ \Delta_{0} = \Delta(G_{e}\,/H)_{0} + \sigma,\,\,\, \Delta_{1}  = \Delta(G_{e}\,/H)_{1} + \sigma,\,\,\, {\rm and}$$
$$\Delta_{\infty} = {\rm Conv}\left(\left\{\left. \frac{v-\pi_\pm(v)}{\beta_{\pm}(v)}\,\right|\,v\in\Cc(1)\setminus e^{\perp}\right\}\right) + \sigma
\,\,\,{\rm if}\,\,C = \PP^{1},$$
where $\pi_\pm$ is the projection defining $\beta_\pm.$
\end{proposition}
\begin{proof}
First, we note by \cite[Lemma 2.4]{Kn} that $\Cc(1)\cap \Vc$ consists exactly
 of images in $N_{\QQ}$ of the $G_{e}$-divisors of $X$, and elements of
$\Cc(1)\cap  e^{\perp}$ (resp. $\Cc(1)\setminus e^{\perp}$) correspond to horizontal (resp. vertical) $G_e$-divisors.

The $G_{e}$-algebra $\KK[X] = A[C,\D]$ admits a presentation as in Theorem \ref{theorem:SL},
where $\D$ is a $\sigma'$-polyhedral divisor for some strongly convex polyhedral cone $\sigma'\subset N_\QQ$.
Using Proposition~\ref{proposition:col.G-div} and Theorem~\ref{theorem:first}, 
we remark that $\sigma\subset N_{\QQ}$ is generated by
$\sigma'(1)$ if $C = \AA^{1}$ and by $\sigma'(1)^{\star}$ and ${\rm deg}\, \D$ if
$C = \PP^{1}$. Hence $\sigma = \sigma'$. Clearly, we have $\Delta_{0} = \Delta(G/H)_{0} + \sigma$
and $\Delta_{1} = \Delta(G/H)_{1} + \sigma$. For the case $C = \PP^{1}$, the description of $\Delta_{\infty}$
follows again from \ref{proposition:col.G-div}. 
\end{proof}

\subsection{Classification of type II}
\

\label{subsec:3.3}
In this subsection, we explain the connection between Demazure roots and affine colored cones for the
type II. Let $e$ be a nonzero vector of $M$. 
Note by \cite[Proposition 3.8]{AL}
an affine $G_{e}$-variety $X$ is not of type {\rm I} if and only if
$X$ is horospherical. 
The following result shows that an affine $G_{e}$-spherical $X$ that is not of type I
can be assumed to be of type II for a larger reductive group.
\begin{lemme}
\label{lemmenotoftypeI}
Let $\Omega$ be a quasi-affine $G_{e}$-horospherical homogeneous space. Assume that the induced $\mathbb{T}$-action on
$\Omega$ is of complexity one. Then there exists an open subset $\tilde{\Omega}$
of the affine closure $\bar\Omega=\Spec\K[\Omega]$, which is a homogeneous space under a reductive group $\tilde{G}_{\tilde{e}} = 
\SL_{2}\rtimes_{\tilde{e}}\tilde{\TT}$ containing $G_{e}$ with the following properties.
\begin{itemize}
\item[(a)] The torus $\TT$ is a subtorus of $\tilde{\TT}$.
\item[(b)] The $\tilde{\TT}$-action on $\tilde{\Omega}$ is of complexity $0$.
\item[(c)]  $\tilde{\Omega}$ is an embedding of the $G_e$-spherical homogeneous space $\Omega$. 
\end{itemize}
Moreover, 
the inclusion map $\Omega\rightarrow\tilde{\Omega}$ naturally identifies the affine $G_{e}$-embeddings of $\Omega$
with the affine $\tilde{G}_{\tilde{e}}$-embeddings of $\tilde{\Omega}$. 
\end{lemme}
\begin{proof}
First of all, $\K[\Omega]$ is finitely generated, so $\bar\Omega=\Spec\K[\Omega]$ is an affine variety and $\codim_{\bar\Omega}(\bar\Omega\setminus\Omega)\ge2$.
 Since the $\SL_2$-action induced by $G_e$ on $\bar\Omega$ is horospherical and horizontal, \cite[Proposition 3.8]{AL} implies
 the existence of the $\tilde{G}_{\tilde{e}}$-homogeneous
space $\tilde{\Omega}\subset \bar\Omega$ that is a $G_e$-equivariant embedding of $\Omega$.

Moreover, let $X$ be an affine $G_e$-embedding of $\Omega$. The $\tilde{G}^{\prime}_{\tilde{e}'}$-action on $X$ is constructed in \cite[Proposition 3.8]{AL}
by extending the torus $\TT\subset G_e$ by the torus $S$ commuting with $G_e$, see \cite[Corollary 3.4]{AL}.  
The torus $S$ corresponds to the grading defined by the decomposition into isotypic components with respect to the $\SL_2$-action. Hence it does not depend on the embedding $X$. 
Therefore, $\tilde{G}^{\prime}_{\tilde{e}'}=\tilde{G}_{\tilde{e}}$ and $\tilde{G}_{\tilde{e}}$-action on $X$
is the extension of $\tilde{G}_{\tilde{e}}$-action on $\tilde\Omega$.
\end{proof}
For the next paragraphs, we let $X$ be an affine $G_{e}$-spherical variety of type II and assume
that $\TT$-action on $X$ is faithful.
\begin{rappel}
\label{marker: semisimpleroot}
Since the torus $\TT$ acts on $X$ with an open orbit, we may represent the $M$-graded algebra $\KK[X]$ as 
the semigroup algebra
$$\KK[\sigma^{\vee}\cap M] = \bigoplus_{m\in\sigma^{\vee}\cap M}\KK\chi^{m},$$
where $\sigma\subset N_{\QQ}$ is a strongly convex polyhedral cone. For any $m\in\sigma^{\vee}\cap M$ 
the character $\chi^{m}:\TT\rightarrow \G_{m}$
is identified with a regular function on $X$. 
Recall that the set of \emph{semisimple roots} \cite[Section 3.4]{Oda} of $\sigma$ is
$$\Rs(\sigma) = \Rt(\sigma)\cap (-\Rt(\sigma)).$$ By \cite[Theorem 2.7]{AL}, we have $e\in\Rs(\sigma)$. 
Denoting by $\rho_{\pm}$ the distinguished ray of $\pm e$,
the $\mathfrak{sl}_{2}$-triplet
on $\KK[X]$ corresponding to the $G_{e}$-action on $X$ is $\{\partial_{-},\delta, \partial_{+}\}$,
where $\partial_{\pm}(\chi^{m}) = \langle m,\rho_{\pm}\rangle \chi^{m\pm e}$ for any $m\in\sigma^{\vee}\cap M$,
and $\delta  = [\partial_{+},\partial_{-}]$. 
\end{rappel}
\begin{lemme}
\label{lemmatoricdiv}
The following assertions hold.
\begin{enumerate}
 \item[(i)] The lattice $L_{\pm}$ of $B_{\pm}$-weights of $\KK(X)$ 
is the sublattice of $M$ generated by $\tau_{\pm}\cap M$,
where $\tau_{\pm}\subset \sigma^{\vee}$ is the dual facet of the ray $\rho_{\pm}\subset\sigma$. 
\item[(ii)] The open $B_{\pm}$-orbit on $X$ is $X_{B_{\pm}} = {\rm Spec}\,\KK[\rho_{\pm}^{\vee}\cap M]$.
\item[(iii)] The open $G_{e}$-orbit on $X$ is the toric $\TT$-variety $X_{\Sigma}$
associated with the fan $\Sigma  = \{\rho_{-}, \{0\}, \rho_{+}\}$.  

\end{enumerate} 
\end{lemme}
\begin{proof}
(i) This follows from the equality ${\rm ker}\,\partial_{\pm} = \KK[\tau_{\pm}\cap M]$.  

(ii) $X_{B_{\pm}}\cong\AA^{1}\times\G_m^{n-1}$, where $n=\dim\TT$, is an open $B_{\pm}$-stable subset of $X$.
It consists of two $\TT$-orbits, namely  $\{0\}\times \G_m^{n-1}$ and $\G_m\times \G_m^{n-1}\cong\TT$. Since a $U_{\pm}$-orbit, which is an affine line,
 cannot be contained in neither $\TT$-orbit, one concludes.
  
(iii) By the above step, the open $B_\pm$-orbit corresponds to the fan $\{\{0\},\rho_\pm\}$, hence the open $G_e$-orbit contains $X_{\Sigma}$. So it remains
to show that $X_{\Sigma}$ is $\SL_{2}$-stable. The coordinate ring of $X_{\Sigma}$ is $\KK[\Lambda\cap M]$, 
where $\Lambda=\rho_{-}^{\vee}\cap\rho_{+}^{\vee}\subset M_{\QQ}$.
Since $\K[X_\Sigma]$ is $\SL_2$-stable, it remains to show that the complement $\Spec(\KK[\Lambda\cap M])\setminus X_\Sigma$ is also $\SL_2$-stable.
The ideal of the complement
$$ I = \bigoplus_{m\in (\Lambda\setminus {\rm lin}(\Lambda))\cap M}\KK\chi^{m}\subset \KK[\Lambda^{\vee}\cap M]$$
is $\partial_{\pm}$-stable, which proves the assertion.
\end{proof}

\begin{remark}
Given an affine horospherical embedding $X$, the associated colored cone is $(\Cc,\F)$, where $\F$ is the set of all colors of $X$, see Definition~\ref{definition:conecol}. 
In particular, for any affine embedding of $X_\Sigma$ the associated colored cone contains exactly one color.
\end{remark}

The main result of this subsection is the following proposition. It can also be obtained from \cite[Theorem 1.1]{AKP}.
\begin{proposition}
\label{proptoric}
Let $X$ be an affine $G_{e}$-spherical variety of type {\rm II}. Assume that the induced $\TT$-action on $X$
is faithful.
\begin{itemize}
\item[(i)]Assume that $X$ is the toric $\TT$-variety corresponding to the strongly convex cone 
$\sigma\subset N_{\QQ}$.
Let $\omega_{\pm}$ be the image of the cone
$$
\sum_{\rho\in\sigma(1)\setminus\{\rho_{\pm}\}}\QQ_{\geq 0}\cdot\rho\subset \sigma
$$
under the projection $N_{\QQ}\rightarrow (V_{\pm})_{\QQ} = {\rm Hom}(L_{\pm},\QQ) =  (N/\langle \rho_{\pm}\rangle)_{\QQ}$,
and $D_{\pm}$ be the $\TT$-divisor on $X$ corresponding to the ray $\rho_{\pm}$.
Then  the associated colored cone 
of $X$ in $(V_{\pm})_{\QQ}$ with respect to
the $B_{\pm}$-action
is $(\omega_{\pm}, D_{\mp})$.
%
\item[(ii)] Conversely, assume that $X$ is an affine embedding of the $G_{e}$-homogeneous space
$X_{\Sigma}$ {\rm (}see \ref{lemmatoricdiv}{\rm )} given by a colored cone $(\Cc, D_{\mp})$, where $D_{\mp}$
is the color of $X$ for the $B_{\pm}$-action. Let $\rho_{\pm}\in N$ be the primitive vector corresponding to the $\TT$-divisor $D_{\pm}$
and consider $\Cc$ as a cone in $N$. Then $X$ is equal to the toric $\TT$-variety ${\rm Spec}\,\KK[\sigma^{\vee}\cap M]$ for
$\sigma=\Cc+\QQ_{\ge0}\rho_\pm$. 
\end{itemize}
\end{proposition}
\begin{proof}
(i) By Lemma \ref{lemmatoricdiv}, the $G_{e}$-divisors of $X$ correspond to elements of 
$\sigma(1)\setminus \{\rho_{-},\rho_{+}\}$ and the image of the unique color (with respect to $B_{\pm}$) is $\rho_{\pm}$.

Assertion (ii) follows directly from (i).
\end{proof}

\section{$G_e$-spherical homogeneous spaces}\label{sec:hom}
In this section we study the correspondence between classes of spherical subgroups of $G_e$ and types of $G_e$-spherical varieties, namely skew, reflexive and horospherical ones. For this, we explicitly describe the $G_e$-spherical homogeneous spaces. We start with the classification of the subgroups of $G_e$ in the following lemma.
\begin{lemma}
\label{lemma:subgroup}
\begin{enumerate}[label*={\rm\Roman*.}]
\item
The group $G_e$ posesses a non-trivial center 
\begin{equation}
Z(G_e)=\left\{\left.
\left(\left(
\begin{matrix}
a&0\\ 0&a^{-1}
\end{matrix}
\right),t\right)
\;\right|\; \chi^{e}(t)=a^{2},a\in\K^{*},t\in\TT\right\}.
\end{equation}
\item Any closed subgroup 
of $G_{e}$
up to conjugation 
is of one of the following types:
\begin{enumerate}[label={(\roman*)}]
\item ``center extension'', i.e. a subgroup $H\subset G_e$ such that the neutral component $H^{\circ}$ is contained in the center $Z(G)$. 
\item quasitorus $Q\subset T\times\TT$ such that $Q^{\circ}\nsubset Z(G_e)$;
\item ``normalizing subgroup'' $N(Q)=Q\cup \tau Q$ such that $Q^{\circ}\nsubset Z(G_e)$.
Such a $\tau$ exists if and only if $Q^\0\cap\SL_2=T$. In latter case we may choose $\tau$  so that $\tau=\left(\left(
\begin{smallmatrix}
0&-1\\ 1&0
\end{smallmatrix}
\right),b\right)$ and  $b^{2d}=1$, where $d$ is the exponent of $Q/Q^\0$;
\item semidirect product $U_+\rtimes Q$, where the quasitorus $Q$ is a subgroup of $T\times\TT$;
\item semidirect product $\SL_2\rtimes Q'$, where the quasitorus $Q'$ is a subgroup of $\TT$.
\end{enumerate}
\item The spherical subgroups of $G_e$ are those with a non-central neutral component, i.e. of types (ii)--(v).
\end{enumerate}
\end{lemma}
\begin{proof}
Let $\left(A,t\right)$ be an arbitrary element in $G_e$. Conjugating by $(B,t^{-1})\in G_e$, we obtain 
\begin{equation}\label{Ge-conj}
(B\cdot\varepsilon_e(t^{-1})(A\cdot\varepsilon_e(t)(B^{-1})),t)=(B\cdot\varepsilon_e(t^{-1})(A)\cdot B^{-1},t).
\end{equation}
It is easy to see that up to conjugation we may assume $A$ to be upper (or lower) triangular for any $\left(A,t\right)$, and diagonal for a semisimple element $\left(A,t\right)\in G_e$. In particular, if $\left(A,t\right)$ is central, then $A$ is diagonal.


Let us study the center $Z(G)$. Consider the conjugation of a diagonal element $z=\left(\left(
\begin{smallmatrix}
a&0\\ 0&a^{-1}
\end{smallmatrix}
\right),b\right)\in G$ by an arbitrary $g=\left(\left(
\begin{smallmatrix}
x&y\\ z&w
\end{smallmatrix}
\right),t\right)\in G$:

\begin{equation}\label{diag-conj}
g\cdot z\cdot g^{-1}=
\left(\left(
\begin{matrix}
axw-a^{-1}yz\chi^{e}(b),&
(a^{-1}-a\chi^{e}(b^{-1}))xy\\
 (a-a^{-1}\chi^{e}(b))wz, &
 a^{-1}xw-ayz\chi^{e}(b^{-1})
\end{matrix}
\right),b\right).
\end{equation}
So, $z\in Z(G)$ if and only if 
$a^{2}=\chi^{e}(b)$. So, the first statement of the lemma is proved.

Let $H\subset G_e$ be a closed subgroup. Assume that all its elements are semisimple. Then the connected component is diagonalizable, so up to conjugation $H^\circ\subset T\times\TT$. 

If $H^{\circ}\subset Z(G)$, then our group is of type (i), i.e. a ``center extension''. 
Otherwise, there exists $h=\left(\left(
\begin{smallmatrix}
a&0\\ 0&a^{-1}
\end{smallmatrix}
\right),b\right)\in H^{\circ}$ such that $h\notin Z(G_e)$. 
Since $H^{\circ}$ is normal in $H$, for any $g=\left(\left(
\begin{smallmatrix}
x&y\\ z&w
\end{smallmatrix}
\right),t\right)\in H$ the conjugation $g\cdot h\cdot g^{-1}$ lies in $H^{\circ}$ and in particular is diagonal. By (\ref{diag-conj}), $xy=wz=0$.
This is possible if $g$ is either of form $\left(\left(
\begin{smallmatrix}
*&0\\ 0&*
\end{smallmatrix}
\right),*\right)$ or of form
$\left(\left(
\begin{smallmatrix}
0&*\\ *&0
\end{smallmatrix}
\right),*\right)$.
Since $Q=H\cap (T\times\TT)$ is a quasitorus, either $H=Q$ or $H=Q\cup \tau Q$ for some $\tau=\left(\left(
\begin{smallmatrix}
0&-a\\ a^{-1}&0
\end{smallmatrix}
\right),b\right)$. In the latter case, $\tau^{2}\in Q$. In addition, if we conjugate $H$ by a diagonal element, then $Q$ is preserved and $\tau$ can be reduced to form $\tau=\left(\left(
\begin{smallmatrix}
0&-1\\ 1&0
\end{smallmatrix}
\right),b\right)$.
Thus, provided an element $q=\left(\left(
\begin{smallmatrix}
a&0\\ 0&a^{-1}
\end{smallmatrix}
\right),t\right)\in Q^\0,$ the commutator
$\tau q\tau^{-1} q^{-1}=q=\left(\left(
\begin{smallmatrix}
\chi^e(t)a^{-2}&0\\ 0&\chi^{-e}(t)a^2
\end{smallmatrix}
\right),1\right)$ is also contained in $Q^\0$. If $Q^\0\cap T$ is trivial, then all elements of $Q^\0$ satisfy $a^2=\chi^e(t)$, so $Q$ is a ``center extension''. Otherwise $Q=T\times Q'$ for a quasitorus $Q'\subset\TT$. Then any $b\in\TT$ such that $b^2\in Q'$ provides a ``normalizing subgroup'' $Q\cup\tau Q$.
Further, if $d$ is the exponent of $Q/Q^\0$, then $b^{2d}\in Q^{\prime\0}$ and there exists $c\in Q^{\prime\0}$ such that $(bc)^{2d}=1$. Thus, multiplying $\tau$ by an element of $T\times\{c\}\subset Q^\0$, we obtain the property $b^{2d}=1$ preserving the form $\tau=\left(\left(
\begin{smallmatrix}
0&-1\\ 1&0
\end{smallmatrix}
\right),b\right)$. We should note that different choices of $b$ may provide different normalizing subgroups.

Now assume there is $h\in H$ that is not semisimple.  By (\ref{Ge-conj}), we may assume up to conjugation that $h$ is upper-triangular. Consider the Jordan decomposition $h=h_s h_u$ into the semisimple part $h_s$ and the unipotent part $h_u$. Since $H$ is closed, $h_u\in H$ and $H$ contains $U_+$.

Let $H$ be contained in the Borel subgroup $\left\{\left(\left(
\begin{smallmatrix}
*&*\\ 0&*
\end{smallmatrix}
\right),*\right)
\right\}.$ Then any element of $H$ is decomposed into an element of $U$ and an element of $T\times\TT.$ So, $H=U\rtimes Q$, where $Q\subset T\times\TT$ is a quasitorus.

Assume now that there exists $h\in H$ that is not contained in the Borel subgroup. Then by direct computation we obtain that the subset 
$$U_+\cdot h\cdot U_+ \cdot h^{-1}\cdot U_+\subset \SL_2$$
is three-dimensional and dense in $\SL_2$. 
Thus, $H=\SL_2\rtimes Q$, $Q\subset\TT$.

Finally, let us describe spherical subgroups. A subgroup $H\subset G_e$ is spherical if and only if the neutral component $H^{\circ}$ is spherical.
The connected subgroup $H^{\circ}$ is spherical if and only if $B\cdot H^{\circ}$ is dense in $G_e$ for some Borel subgroup $B$. Since $B\subset G_e$ is of codimension one, the latter condition is equivalent to $H^\circ\nsubset B$. So, a subgroup $H\subset G_e$ is spherical if and only if $H^{\circ}\nsubset Z(G_e)$.
\end{proof}

\begin{remark}
In fact, normal closed subgroups of $G_{e}$ are exhausted by:
\begin{enumerate}
\item central quasitori $Q\subset Z(G_e)$;
\item semidirect products $\SL_2\rtimes Q'$, where $Q'\subset\TT$.
\end{enumerate}
\end{remark}

\subsection{Classification of type I} 
Here we compute homogeneous spaces $X=G_e/H$ of type I and provide the corresponding combinatorial data in terms of AL-colorings. By the following remark, we may assume that $\dim\TT\le 2$. Moreover, if $\dim\TT<2$, we may extend $G_e$ and $X$ to a direct product with a torus and restrict to the case $\dim \TT=2$, so $\dim X=3$. In fact, the case $\dim\TT<2$ corresponds to the situation $v_0\parallel v_1$ described in Theorem~\ref{th:sph-subgroups-I}(2,4).
\begin{remark}\label{rem:T-reduction}
Let a $G_e$-spherical homogeneous space $X=\Spec A[C,\D]$ be as in Remark~\ref{marker:remarkG/H}.
 Assume that $\dim\TT>2$. Then there exists a decomposition $\TT=\TT'\times\TT''$ with induced lattice decompositions $M=M'\oplus M''$, $N=N'\oplus N''$ enjoying following properties: 
 \begin{itemize}
  \item $M'=N^{\prime\prime\perp},M''=N^{\prime\perp},$
 \item $v_0,v_1\in N'$, $\dim N'=2,$
 \item $N''\subset e^{\perp}$, or equivalently, $e\in M'$.
 \end{itemize}
Then $\Delta_0,\Delta_1\subset N^\prime_\QQ$. So, we may consider $\D$ as the polyhedral divisor in $N^\prime_\QQ$ and denote it by $\D'$.
Therefore,
$A[C,\D]=A[C,\D']\otimes \K[\TT'']$ and $X'=\Spec A[C,\D']$ is a $\TT'$-variety of complexity 1.
In addition, by \ref{rap:der-Q}, 
\[\partial_\pm(Q\chi^{m'+m''})=\partial_\pm(Q\chi^{m'})\chi^{m''}\quad\mbox{for all}\quad m'\in M',\,m''\in M''.\]
Thus, $A[C,\D']$ is stable under the action of the subgroup $G_e^{\prime}=\SL_2\rtimes \TT'\subset G_e$, $X'=\Spec A[C,\D']$ is a $G_e^{\prime}$-spherical homogeneous space, and $X=X'\times\TT''$. Moreover, the condition $N''\subset e^{\perp}$ ensures that the $\TT''$-action  on $X'$ is trivial, so
 $G_e^\prime\times\TT''$ acts on $X'\times\TT''$ componentwise.
This reduction preserves the conditions that the $\TT$-action is faithful of complexity one and that $e\neq0$.
\end{remark}
\begin{rappel}
Let $H\subset G_e$ be a spherical subgroup, whose connected component is a two-dimensional torus and such that the $\TT$-action on $G_e/H$ is faithful.
This corresponds to cases (ii) and (iii) of Lemma~\ref{lemma:subgroup}, i.e. to homogeneous spaces of type I, since cases (iv) and (v) induce a horospherical one.
Then $H^{\0}$ is the image of an embedding $\G_m^{2}\to T\times \TT$ given by a triple of characters $(\chi^f,\chi^{\lambda_1},\chi^{\lambda_2})$ in $\mathbb{X}(\G_m^{2}).$ Thus, 
\[
H^\0=\left\{\left.\left(\left(\begin{matrix}
   \chi^f(t) & 0 \\
   0 & \chi^{-f}(t)
 \end{matrix}\right), (\chi^{\lambda_1}(t),\chi^{\lambda_2}(t))\right)\;\right|\;t\in\G_m^{2}\right\}.
\]
 Since it is an embedding, two cases are possible:
\begin{enumerate}
\item[``$\lambda_1\nparallel\lambda_2$'':] The vectors $\lambda_1$ and $\lambda_2$ are linearly independent. Then there exist $k,\alpha_1,\alpha_2\in\ZZ$ such that $kf=\alpha_1\lambda_1+\alpha_2\lambda_2$. We choose $k\in\ZZ_{>0}$ to be minimal\footnote{The constant $k$ corresponds to $r$ in \cite[Theorem 4.2]{AL}.}.
 By construction, $\alpha=(\alpha_1,\alpha_2)\in M.$
\item[``$\lambda_1\parallel\lambda_2$'':] The vectors $\lambda_1$ and $\lambda_2$ are linearly dependent. 
Then we may assume up to a choice of coordinates in $\G_m^2$ that
$f=(1,0)$, $\lambda_1=(0,-\alpha_2)$, $\lambda_2=(0,\alpha_1)$, where $\mathrm{gcd}(\alpha_1,\alpha_2)=1$.
Moreover, the element $\alpha=(\alpha_1,\alpha_2)\in M$ is such that $\chi^{\alpha}(\chi^{\lambda_1}(t),\chi^{\lambda_2}(t))=1$ for any $t\in\G_m^{2}$.
\end{enumerate}

If $H$ is a quasitorus, we denote it by $H=Q$, otherwise $H=N=Q\cup\tau Q$. By Lemma~\ref{lemma:subgroup}(iii), the latter case is possible if and only if 
$T\subset H$, i.e. $\lambda_1\parallel\lambda_2$.
The representation $G_e=\left\{\left(\left(\begin{smallmatrix}
  a & b \\
   c & d
 \end{smallmatrix}\right), (s_1,s_2)\right)\right\}$ gives us the coordinates on $G_e$ as follows:
 \[\K[G_e]=\K[a,b,c,d,s_1,s_1^{-1},s_2,s_2^{-1}]/(ad-bc-1).\]
\end{rappel}
These preparations allow us to formulate the description of spherical subgroups for type I.

\begin{theorem}\label{th:sph-subgroups-I}
Let $H$ be a spherical subgroup of $G_e$ such that $G_e/H$ is not horospherical and the $\TT$-action on $G_e/H$ is faithful of complexity one.
Then $H$ is either a torus or a ``normalizing subgroup'' of a torus as in Lemma~\ref{lemma:subgroup}(ii,iii), and $M=\bangle{\alpha,e}$.

Summarizing, $H$ is conjugate to one of following cases, the first case corresponding to $\lambda_1\nparallel\lambda_2$, the others to $\lambda_1\parallel\lambda_2$. We use notation $x=ab,y=cd,z=ad+bc, e=(e_1,e_2), s^\alpha=s_1^{\alpha_1}s_2^{\alpha_2}$.
\begin{enumerate}
\item $H=Q_1$ is a two-dimensional torus such that $Q_1\cap T$ is the identity. Then 
\[\K[G_e/Q_1]=\K[a^{k-i}c^is^{-\alpha},b^{k-i}d^is^{\alpha},ab,ad,bc,cd\,|1\le i\le k]/(ad-bc-1).\]
\item $H=Q_2=T\times Q'$, where $Q'\subset\TT$ is a 1-parameter subgroup. Then 
\[\K[G_e/Q_2]=\K[x,y,z,s^\alpha,s^{-\alpha}]/(z^2-4xy-1).\]
\item $H=N_1=T\times Q'\cup\tau_1(T\times Q')$, where $Q'\subset\TT$ is a 1-parameter subgroup, and $\tau=$ the right multiplication by $\tau_1$ sends $(x,y,z,s^\alpha)$ to
$(-x,-y,-z,-s^\alpha).$ Then
\[\K[G_e/N_1]=\K[xs^{\alpha},ys^{\alpha},zs^{\alpha},s^{2\alpha},s^{-2\alpha}]/(z^2-4xy-1).\]
\item $H=N_2=T\times Q'\cup\tau_2(T\times Q')$, where $Q'\subset\TT$ is a 1-parameter subgroup, and the right multiplication by $\tau_2$ sends $(x,y,z,s^\alpha)$ to
$(-x,-y,-z,s^\alpha).$ Then
\[\K[G_e/N_2]=\K[x^2,y^2,z^2,xy,xz,yz,s^\alpha,s^{-\alpha}]/(z^2-4xy-1).\]
\end{enumerate}

Let us take the basis $\{\nu_\alpha,\nu_e\}$ in $N$ dual to the basis $\{\alpha,e\}$ in $M$. Thus, for any $(m_1,m_2)$ in the standard basis of $M$, 
\[\nu_\alpha(m_1,m_2)=\varepsilon(e_2m_1-\alpha_2m_2),\;	 \nu_e(m_1,m_2)=\varepsilon(-e_1m_1+\alpha_1m_2),\]
where $\varepsilon=\left|\begin{smallmatrix}\alpha_1&e_1\\ \alpha_2&e_2\end{smallmatrix}\right|\in\{1,-1\}.$ In Table~\ref{tab:hom} we describe spherical data of all cases.
Note that subgroups of types $Q_1,Q_2$ give reflexive homogeneous spaces, whether $N_1,N_2$ give skew ones. In the former case $v_0,v_1$ can be swapped, and in the latter case the coefficient $h$ in $v_0$ is odd for $N_1$ and even for $N_2$.
An isomorphism $G_e/H\cong\Spec A[C,\D]$ is defined by images of $t,\chi^e,\chi^\alpha$ generating the fraction field of $A[C,\D]$; we have two choices for images of $t,\chi^e$ for $Q_2$ and $N_2$, and the image of $\chi^\alpha$ is defined up to a scalar multiple.  
\begin{table}[h!]
\centering 
\begin{tabular}{r|c|c|c|c|}
&$Q_1$ & $Q_2$ & $N_1$ &  $N_2$  \\ \hline
$K[G/H]^{U_-}$ &$\K[a^{k}s^{-\alpha},b^{k}s^{\alpha},ab]$&$\K[x,s^\alpha,s^{-\alpha}]$ & $\K[xs^\alpha,s^{2\alpha},s^{-2\alpha}]$ &  $\K[x^2,s^\alpha,s^{-\alpha}]$  \\ \hline
$K[G/H]^{U_+}$&$\K[c^{k}s^{-\alpha},d^{k}s^{\alpha},cd]$&$\K[y,s^\alpha,s^{-\alpha}]$&$\K[ys^\alpha,s^{2\alpha},s^{-2\alpha}]$ &  $\K[y^2,s^\alpha,s^{-\alpha}]$\\ \hline
$B_-$-divisors &$\V(a^{k}s^{-\alpha}),\V(b^{k}s^{\alpha})$&$\V(x,z\pm1)$ & $\V(xs^\alpha)$ &  $\V(x^2)$  \\ \hline 
$B_+$-divisors &$\V(c^{k}s^{-\alpha}),\V(d^{k}s^{\alpha})$&$\V(y,z\pm1)$ & $\V(ys^\alpha)$ &  $\V(y^2)$  \\ \hline
$B_-$-colors&$-k\nu_\alpha-\nu_e,\,-\nu_e$&$-\nu_e,\,-\nu_e$&$-\nu_e$&$-\nu_e$\\ \hline
$B_+$-colors&$\nu_e,\,k\nu_\alpha+\nu_e$&$\nu_e,\,\nu_e$&$\nu_e$&$\nu_e$\\ \hline
$v_0,v_1$&$k\nu_\alpha+\nu_e,\,-\nu_e$&$\nu_e,\,-\nu_e$&$\frac{h}{2}\nu_\alpha+\frac{1}{2}\nu_e,\,-\nu_e$&$\frac{h}{2}\nu_\alpha+\frac{1}{2}\nu_e,\,-\nu_e$\\ \hline
$t$&$ad$&$\frac{1+ z}{2},\frac{1- z}{2}$&$z^2$&$z^2$\\ \hline
$\chi^e$&$\frac{d}{b}$&$\frac{z+1}{2x},\frac{z-1}{2x}$&$\frac{z}{2x}$&$\frac{z}{2x}$\\ \hline
$\chi^\alpha$&$d^ks^\alpha$&$s^\alpha$&$zs^\alpha$&$s^\alpha$\\ \hline
\end{tabular}
\caption{Description of $G_e$-homogeneous spaces}\label{tab:hom}
\end{table}
\end{theorem}
\begin{proof}
By a direct computation we have
\[\K[G_e/Q_1^\0]=\K[a^{k-i}c^is^{-\alpha},b^{k-i}d^is^{\alpha},ab,ad,bc,cd\,|0\le i\le k]/(ad-bc-1),\]
\[\K[G_e/Q_2^\0]=\K[G_e/N_{1}^\0]=\K[G_e/N_{2}^\0]=\K[x,y,z,s^\alpha,s^{-\alpha}]/(z^2-4xy-1).\]
Since the $\TT$-action is effective, the lattice of weights, equal to $\bangle{\alpha,e}$ in both cases, coincides with $M$.

 Let us prove that  $Q_{i}=Q_{i}^\0$, $i=1,2$. For the sake of clarity, we allow in the case of $Q_1$ the basis vectors $\alpha$ and $e$ be equal to $(1,0)$ and $(0,1)$ respectively. Thus,  using notation from ``$\lambda_1\nparallel\lambda_2$'', $kf=\lambda_1$, so we may further assume $f=(1,0)$ and $\lambda_2=(0,1)$. If $Q_1$ is not connected, then we have a decomposition $Q_1=Q_1^\0\times\bangle{\zeta}$,
where $\zeta\in Q_1$ has a finite order $r$.
So, 
\[Q_1^\0=\left\{\left.\left(\left(\begin{smallmatrix}
  t_1 & 0 \\
   0 & t_1^{-1}
 \end{smallmatrix}\right)(t_1^k,t_2)\right)\;\right|\;(t_1,t_2)\in\G_m^2\right\},\;
\zeta=\left(\left(\begin{smallmatrix}
  \zeta_r^{r_0} & 0 \\
   0 & \zeta_r^{-r_0}
 \end{smallmatrix}\right),(\zeta_r^{r_1},\zeta_r^{r_2})\right),\]
 where $\zeta_r$ is an $r$th primitive root of unity.
It is easily checked that the element $(\zeta_r^{r_1-kr_0},1)\in\TT$ acts trivially on $G_e/Q_1$, so $r_1\equiv kr_0 \pmod{r}$.
Then $\zeta\in Q_1^\0$ and $Q_1=Q_1^\0$.

Let us now treat the case $Q_2$. The intersection $\TT\cap Z(G_e)=\{t\;|\;\chi^e(t)=1\}$ is a 1-parameter subgroup, because $e$ is primitive. 
Since the $\TT$-action is effective, $Q_2\cap \TT\cap Z(G_e)$ is trivial. Therefore, $Q_2\cap \TT$ is finite, which is impossible, or a 1-parameter subgroup $Q'$.
Thus, $Q_2=T\times Q'$ is also connected.

Computation of $\K[G_e/H]$, $\K[G_e/H]^{U_\pm}$, and $B_\pm$-divisors is straightforward.
We compute the colors, using basis $\{\nu_\alpha,\nu_e\}$ dual to $\{\alpha,e\}\in M$. 
For each case we can choose a pair of functions of weights proportional to $\alpha$ and $e$ respectively. For example, in the case $Q=Q_1$ we take $a^{k}s^{-\alpha}$ and $ab$.
They have zeroes of order $k$ and $1$ on $\V(a^ks^{-\alpha})$ respectively, so $\varrho(\V(a^ks^{-\alpha}))=-k\nu_\alpha-\nu_e$.
Similarly we compute all other colors,   using Proposition~\ref{proposition:col.G-div} if necessary.

We cannot extract $v_0$ from colors in cases of $N_1,N_2$, and the only conditions we have is $2v_0(e)=1, 2v_0\in N$.
Thus, $v_0=\frac{h}{2}\nu_\alpha+\frac{1}{2}\nu_e$, where $h\in\ZZ$. Since $v_0$ can be shifted by an integer vector,
we have two possible cases: either $h$ is odd or even.
They correspond respectively to $N_1$ and $N_2$.

To obtain an explicit isomorphism $\phi\colon X\stackrel{\sim}{\longrightarrow}\Spec A[C,\D]$ as in Theorem~\ref{theorem:SL}, we use following observations.
First,  we have $\phi^*(\K[t])=\K[X]^\TT$, thus $\phi^*(t)=pz+r$ for some $p\in\K^\times$, $r\in\K$. Second, $\partial_+(t)\partial_-(t)$ is equal to $t(t-1)$ and $4t(t-1)$ in reflexive and skew cases respectively, thus we obtain two possibilities for $\phi^*(t)$. Computing $\partial_+(t)$, we also find two corresponding possibilities for $\phi^*(\chi^e)$. Further, $\phi^*(\chi^\alpha)$ is obtained from considering the homogeneous component $A_\alpha\chi^\alpha\subset A[C,\D]$ up to a scalar multiple. 
Considering the component $A_{-\alpha}\chi^{-\alpha}$ in cases $H=Q_1$ and $H=N_1$ and the component $A_{-2e}\chi^{-2e}$ in case $H=N_2$, we eliminate one of possibilities for $\phi^*(t)$.
For example, for $H=Q_1$,  we have
\begin{align*}
A_\alpha\chi^\alpha=\K[t]\chi^\alpha&\stackrel{\phi^*}{\longrightarrow}\K[ad]d^ks^\alpha,\\
A_{-\alpha}\chi^{-\alpha}=\K[t]t^k\chi^{-\alpha}&\stackrel{\phi^*}{\longrightarrow}\K[ad]a^ks^{-\alpha},
\end{align*}
thus $\chi^\alpha$ maps to a scalar multiple of $d^ks^\alpha$ and $\K[t]t^k$ to $\K[ad](ad)^k$.
On the other hand, in case $Q_2$ both choices of images of $t$ and $\chi^e$ are valid, and the involutions  $\tau_1,\tau_2$ permute them. 
\end{proof}

\begin{remark}
In conclusion, taking into account Remark~\ref{marker:remarkG/H}, the type of a $G_e$-spherical embedding  of a homogeneous spherical space $G_e/H$ is defined by $H$ as follows.
The cases (ii), (iii), and (iv--v) of Lemma~\ref{lemma:subgroup} induce reflexive, skew, and horospherical varieties respectively.
\end{remark}

\subsection{Classification of type II} 
By Proposition~\ref{proptoric}, a $G_e$-horospherical homogeneous space $X$ of type II is toric corresponding to the pair of non-collinear rays $\rho_-,\rho_+$.
Since $\rho_\pm(e)=\mp1,$ there exists a decomposition $\TT=\TT'\times\TT''$ as in Remark~\ref{rem:T-reduction} such that the group $G_e=G_e^\prime\times\TT''$ acts on $X=X'\times\TT''$ componentwise and $\dim\TT'\le 2$. Since the case $\dim\TT\le1$ is not possible, we assume that $\dim\TT=2$, hence $\dim X=2$, and obtain the following description of spherical subgroups.

\begin{proposition}
Let $X=G_e/H$ be a $G_e$-spherical homogeneous space of type II. Assume that $\dim\TT=2$, $X$ is toric, and the $\TT$-action on $X$ is faithful. Then $H$ is conjugate to
\[
\left(U_+\rtimes\bangle{\zeta}\right)\rtimes\TT\subset\SL_2\rtimes\TT,
\]
where $\zeta\in T\subset\SL_2$ is an element of finite order $r=[N:\bangle{\rho_-,\rho_+}]$.
Conversely, for each $r$ such subgroup $H$ provides a $G_e$-spherical homogeneous space of type II.
\end{proposition}
\begin{proof}
Recall \cite[Example 2.9]{AL} identifying $N$ with $\ZZ^2$ so that $\rho_-=(1,0), \rho_+=(r,s)$, where $0\le r\le s$ and $\gcd(r,s)=1.$
As explained in the example, it implies that $e=(1,-1)$ and $s=r+1$. Given the morphism
\[
\phi\colon\TT\to\AA^{s+1},\; (t_1,t_2)\mapsto(t_1^{s+1},t_1^st_2,\ldots,t_2^{s+1}),
\]
we obtain the isomorphism $X\cong X_\Sigma=\overline{\phi(\TT)}\setminus\{0\}$. 
The action of $G_e=\SL_2\rtimes_e\TT$ on $X_\Sigma$ is induced by the canonical $\SL_2$-action on the simple $\SL_2$-module $\bangle{t_1^{s+1},t_1^st_2,\ldots,t_2^{s+1}}$ of homogeneous forms of degree $s+1$ and by the natural $\TT$-action on $\phi(\TT)$.

Let $X=G_e/H$ and $H'$ be the stabilizer of the element $((1,0,\ldots,0))\in\phi(\TT)$. Since $H'$ is conjugate to $H$, it is enough to describe $H'$. It is easily seen that 
$H'=(U_+\rtimes\bangle\zeta)\rtimes\TT$, where $\zeta=\left(\begin{smallmatrix}
  \zeta_{s+1} & 0 \\
   0 & \zeta_{s+1}^{-1}
 \end{smallmatrix}\right)\in T\subset\SL_2$ is an element of order $s+1$.
\end{proof}

\section{Demazure roots and affine spherical varieties}\label{sec:dem}
The existence of normalized $\G_{a}$-actions on an affine
spherical variety imposes restrictive conditions on its geometric structure. We study some of these conditions in the framework 
of the Luna--Vust theory. In particular, we define the notion of Demazure roots of an affine spherical
embedding $X$ (see \ref{rootsphere}), which describes  the normalized $\G_{a}$-actions on $X$. 
Section \ref{subsec:4.2} is devoted to the example of affine $G_{e}$-spherical varieties, 
where $G_{e}$ is the reductive group $\SL_{2}\rtimes_{e}\TT$.
\subsection{General case} 
\

We denote by $G$ a connected reductive group and fix a Borel subgroup $B\subset G$.
We also let $U$ be the unipotent radical of $B$. We begin with a preliminary lemma.
\begin{lemme}
\label{lemma: LNDBweight}
Let $Z$ be a quasi-affine variety endowed with a $G$-action. Let $\partial$ be
an {\rm LND} on $\KK[Z]$. Then the  $\G_{a}$-action corresponding to $\partial$
is normalized of degree $\chi$ if and only if
$$g\cdot \partial(f) = \chi(g)\cdot \partial(g\cdot f)$$
for all $f\in \KK[Z]$ and $g\in G$. 
If the latter condition holds, then $\partial$ sends $B$-eigenvectors
into $B$-eigenvectors. Furthermore, if $Z$ contains an open $B$-orbit and $\partial$ is nonzero, 
then $\partial$ is uniquely determined
by the character $\chi$ up to the multiplication by a nonzero constant.
\end{lemme}
\begin{proof}
See \cite[Section 1.1]{FKZ} for the existence and properties of the LND $\partial$.
The first claims are obtained similarly to the proof of \cite[Lemma 2.2]{FZ1}.
Moreover, for all $b\in B$ and $B$-eigenvector $f\in\KK[Z]$ with $B$-weight $\chi_{f}$, one has 
$$b\cdot \partial(f) = \chi(b)\cdot \partial(b\cdot f) = (\chi\cdot\chi_{f})(b)\cdot\partial(f).$$
Hence, $\partial$ sends $B$-eigenvectors into $B$-eigenvectors. 

Assume that $Z$ contains an open $B$-orbit. Passing to the normalization,
we may suppose that $Z$ is normal. Since the semisimple $G$-module $\KK[Z]$ has no multiplicities,
the subalgebra $\KK[Z]^{U}$ 
is a $\partial$-stable normal finitely generated semi-group algebra.
Furthermore, the restriction $\partial'$ of $\partial$ on $\KK[Z]^{U}$ is a homogeneous LND.
Since every simple submodule of $\KK[X]$ is generated by $G.f$ for some $B$-weight $f\in \KK[X]$, $\partial$
is uniquely determined by $\partial'$. One concludes by \cite[Corollary 2.8]{Li}.
\end{proof}
In the sequel, we fix an affine embedding $X$ of a spherical homogeneous space $G/H$
given by a colored cone $(\Cc, \F_{Y})$. We denote by $\F_{0}$ the set of colors
of $X$ and by $\Gamma$ the cone generated by $\varrho(\F_{0})\cup \Cc$. We
recall that $\Gamma^{\vee}$ is the cone of $B$-weights of $\KK[X]$ (see Lemma \ref{lemma:affcone}).
\begin{rappel}
An embedding $Z$ of $G/H$ is called \emph{elementary}
if $Z$ has two orbits so that the complement of the open orbit 
is a divisor. In particular, every elementary embedding is smooth.
There is a natural bijection between non-trivial normalized 
discrete $G$-invariant valuations on $\KK(G/H)$ and elementary
embeddings of $G/H$ (see \cite[2.2]{BP}). 

Let $v\in\Gamma(1)$. We say that $v$ \emph{corresponds to a $G$-divisor} of $X$
if there exists a $G$-divisor $D\subset X$ such that $v = \varrho(D)$.
In the sequel, we make the following convention. If $v$ corresponds to a $G$-divisor of $X$, then $X_{v}$
 denotes the associated elementary embedding of $G/H$ given by the colored cone $(\QQ_{\geq 0}\cdot v, \emptyset)$.
Otherwise, $ v \in \QQ_{\ge0}\cdot\varrho(D_{0})$ for some $D_{0}\in\F_{0}$ (see \cite[Lemma 2.4]{Kn}) 
and we  let $X_{v} = G/H$.
\end{rappel}
Let us introduce the main definition of this subsection.
\begin{definition}
\label{rootsphere}
A character $\chi$ of $G$ is called a \emph{Demazure root} of $X$
if $\chi\in \Rt(\Gamma)$ and the following holds.
The embedding $X_v$, where $v = \rho_{\chi}$ is the distinguished ray of $\chi$ (see \ref{marker: Demazure}), admits  a normalized $\G_a$-action of degree $\chi$ and is homogeneous under the induced $\G_{a}\rtimes_{\chi}G$-action.

We denote by $\Rt(X)$
the set of Demazure roots of $X$. A root $\chi\in\Rt(X)$ is said to be 
\emph{exterior} if $\rho_\chi$ corresponds to a $G$-divisor. Otherwise, $\chi$
is called \emph{interior}. We denote by $\Rte(X)$ (resp. $\Rti(X)$)
the set of exterior (resp. interior) roots of $X$. Clearly, we have $$\Rte(X)\cap \Rti(X) = \emptyset\,\,\, 
{\rm and}\,\,\,\Rt(X) = \Rte(X)\cup \Rti(X).$$    
\end{definition}
\begin{rappel}
Two nonzero LNDs $\partial,\partial'$ on $\KK[X]$ are called equivalent if $\partial = \lambda\cdot\partial'$
for some $\lambda\in\KK^{*}$. We denote by $\LG(X)$ the set of equivalence
classes of nonzero LNDs on $\KK[X]$ corresponding to normalized $\G_{a}$-actions. 
Since $\deg \partial = \deg\partial'$ for equivalent $\partial$ and $\partial'$, the degree is well defined on $\LG(X)$. 
It also coincides with the degree of the homogeneous LND on $\KK[X]^{U}$ induced by any $\partial'\in [\partial]$.
\end{rappel}

The following theorem allows to reduce the classification of normalized $\G_{a}$-actions
on $X$ to the case when $X$ is
 either quasi-affine homogeneous or elementary, and the induced action of $\G_{a}\rtimes_{\chi}G$ on $X$ is transitive.
\begin{theorem}
\label{theoroot}
The map 
\begin{equation}\label{eq:iota}
\iota:  \LG(X)\rightarrow \Rt(X),\,\,\, [\partial]\mapsto {\rm deg}\,\partial 
\end{equation}

is well defined and bijective. Furthermore, let $\partial\in\LND_G(X)$ and $\chi=\deg \partial\in\Rt(X)$ have the distinguished ray $v$. 
Then the following hold. 
\begin{itemize} 
\item[(i)] 
The open orbit of the induced $\G_a\rtimes_\chi G$-action on $X$ is  $X_v$.
\item[(ii)] We have $-v\in\Vc$.
\item[(iii)] 
The kernel ${\rm ker}\, \partial$ is a finitely generated
subalgebra. 
\end{itemize}
\end{theorem}
\begin{proof}
Consider $[\partial]\in\LG(X)$. Let us show that $\chi = {\rm deg}\, \partial\in\Rt(X)$.
By Lemmata~\ref{lemma:affcone}, \ref{lemma: LNDBweight}, and \cite[Theorem 2.7]{Li}, 
we have $\chi\in\Rt(\Gamma)$.
Choose a $B$-eigenvector $f\in \KK[X]$ such that $\chi_{f}\in\relint
(\Gamma\cap\rho^{\perp}_{\chi})\cap L$, and let $B.x$ be the open $B$-orbit of $X$.
Since $f\in {\rm ker}\, \partial$, the subset $X_{f}=X\setminus \div(f)$ is stable under the $\G_{a}$-action
induced by $\partial$. 

Assume first that $v = \rho_{\chi}\in\Gamma(1)$ does not correspond to a $G$-divisor of $X$.
Then we may consider the subset of colors
$$\F' = \{D'\in \F_{0}\,\,|\,\, \varrho(D')\in\rho_{\chi}\}.$$
Using the relation
$${\rm div}\, f = \sum_{D\in\F_{0}\cup\mathcal{P}}\langle \chi_{f},\varrho(D)\rangle\cdot D,$$
where $\mathcal{P}$ is the set of $G$-divisors of $X$,
we obtain the equality
$$X_{f} =  X\setminus \bigcup_{\mathcal{P}\cup\F_0\setminus\F'}D.$$
The complement to $ (G/H)\cap X_f$ in $X_f$ is of codimension at least two,
hence the subset
$$Z = \G_{a}\star\left(X_f\setminus  (G/H)\right)$$
is  $\G_a$-stable and of codimension at least one. Thus, $G/H = G\cdot (X_{f}\setminus Z)$ is also $\G_{a}$-stable. Hence $\chi\in \Rti(X)$ and (i) is proven for this case.

Now assume that $v = \rho_{\chi}\in\Gamma(1)$ corresponds to a $G$-divisor $D_v$ of $X$. 
Then
$$X_{f} =  X\setminus \bigcup_{(\mathcal{P}\cup\F_0)\setminus\{D_v\}}D.$$
So, the open subset $G\cdot X_{f}$ is $\G_{a}$-stable and its unique $G$-divisor is $D_v$. 

Let us prove by contradiction that $\G_{a}\star D_v$ is dense in $X$.
If $\G_{a}\star D_v$ is not dense, then $G/H$ is $\G_{a}$-stable. Remark
that the dual cone $\Gamma_{0}$ of the $B$-weights of $\KK[G/H]$ is the subcone of $\Gamma$ generated by $\varrho(\F_{0})$. By considering the induced $\LND$ on $\KK[G/H]^{U}$ 
(see \ref{lemma: LNDBweight}),
which has same degree $\chi$, we have $\chi\in\Rt(\Gamma_{0})$. Let 
$v'$ be the distinguished ray of $\chi$ with respect to $\Gamma_{0}$.
The inclusion $$\KK[X]^{U}\cap \KK(X)^{\partial}\subset \KK[G/H]^{U}\cap \KK(X)^{\partial}$$
yields
the inclusion of facets 
$v^{\perp}\cap\Gamma^{\vee}\subset v'^{\perp}\cap \Gamma^{\vee}_{0}.$
Hence $v^{\perp} = v'^{\perp}$. Since $v,v'$ are primitive
vectors and $\Gamma$ is strongly convex, we have $v = v'\in \Gamma(1)\cap\Gamma_{0}(1)$.
Therefore $v$ does not correspond to a $G$-divisor of $X$, a contradiction.
Hence $\G_{a}\star D_v$ is dense in $X$. 

Let $G.y$ be the orbit which  closure 
in $G\cdot X_{f}$ is $D_v$. Such an orbit $G.y$ exists due to the fact that $X$ has a finite number of 
orbits. Then the subset 
$$Z_{0} = G\cdot X_{f}\setminus (G/H\cup G.y)$$
is closed of codimension  at least two.
So, $Z_1=\overline{\G_{a}\star Z_{0}}$ is a $\G_{a}\rtimes_{\chi}G$-stable subset of codimension at least
one. Hence $Z_1\cap G/H = \emptyset$.
If $Z_1\cap G.y\neq \emptyset$, then $Z_{1}$ contains $D_v$,
yielding a contradiction. Therefore, $Z_{1}=Z_0$
and $X_{v}$ is $\G_{a}$-stable.    
Since $D_v$ is not $\G_a$-stable, $\G_{a}\rtimes_{\chi}G$ acts homogeneously on $X_{v}$ and 
$\chi\in \Rte(X)$. This proves (i) and shows that $\iota$ in (\ref{eq:iota}) is well defined. Note that
the injectivity of $\iota$ follows from Lemma \ref{lemma: LNDBweight} and the surjectivity
is a consequence of (i).


(ii) 
We define the discrete $G$-invariant valuation $w$ 
on $\KK(G/H)$ by letting
$$w(f) = - \max\{i\in\mathbb{N}\,|\, \partial^{i}(f)\neq0\,\}$$
for every nonzero $f\in \KK[X_{v}]$. Further, if $f$ is a $B$-eigenvector, then $w(f) = \langle \chi_{f}, -v\rangle$.
Hence $w = -v\in\linea\Vc$. 

(iii) This follows from the fact that the $({\rm ker}\,\partial)^{U}$ is finitely generated
(see \cite[Theorem 1.2]{Ku} and \cite[Corollary 2.11]{Li}) and that ${\rm ker}\,\partial = G\cdot ({\rm ker}\,\partial)^{U}$. 
\end{proof}
By the previous result we may introduce the following definition.
\begin{definition}
A non-trivial normalized $\G_{a}$-action of degree $\chi$ on $X$ is said to be 
\emph{exterior} (resp. \emph{interior}) if $\chi\in\Rte(X)$ (resp. $\chi\in \Rti(X)$) 
\end{definition}
As a direct consequence of Theorem \ref{theoroot}, we have the following results.
\begin{corollaire}
If $G$ is semisimple, then $\Rt(X) = \emptyset$.
If $\Vc$ is strongly convex, then $\Rte(X) = \emptyset$. 
\end{corollaire}
\begin{proof}
This follows from the fact that $\Rt(X)\subset \Rt(\Gamma)$ and from the part (ii) of Theorem \ref{theoroot}. 
\end{proof}
\begin{corollaire}
\label{fixedpointcor}
If $X$ has three orbits including a fixed point $y$, then  $\Rte(X) = \emptyset$.
\end{corollaire}
\begin{proof}
Take $\chi\in\Rte(X)$. By Theorem \ref{theoroot}, $X$ admits a normalized $\G_{a}$-action on $X$ that
induces a transitive $\G_{a}\rtimes_{\chi}G$-action on $X_{v}$, where
$v\in\Cc(1)\cap\Vc$. 
Therefore, $X\setminus X_{v}$ is $\G_{a}$-stable and
$y$ is the unique fixed point for the $\G_{a}$-action. This yields a contradiction (see \cite{Bi}). 
\end{proof}
\begin{example}
Consider the spherical action of $G = {\rm GL}_{2}\times {\rm GL}_{2}$  on the variety $X = {\rm M}_{2\times2}$
of $2\times 2$ matrices  via the relation $(P,Q)\cdot M = PMQ^{-1}$
for all $(P,Q)\in G$ and $M\in X$. Then $X$ has three orbits:
$$G\cdot\begin{pmatrix} 0  & 0 \\ 0   &   0 \end{pmatrix},\,\,\,
G\cdot\begin{pmatrix} 1  & 0 \\ 0   &   0 \end{pmatrix},\,\,\,
G\cdot\begin{pmatrix} 1  & 0 \\ 0   &   1 \end{pmatrix},$$
where the first is the fixed point. Hence by Corollary \ref{fixedpointcor},
every exterior normalized $\G_{a}$-action on $X$ is trivial. 
\end{example}
The set $\Rt(X)$ of Demazure roots of an affine spherical variety $X$ is defined in geometric terms.
A natural question is to know whether $\Rt(X)$ is described in combinatorial terms. More precisely,
we propose the following problem.
\begin{question}
\label{question}
For any affine spherical variety $X$, is the set $\Rt(X)$ equal to the intersection of
a finite union of polyhedra in $L_{\QQ}$ with a sublattice of $L$? 
\end{question}

\subsection{Description for the reductive group $G_{e}$}
\

\label{subsec:4.2}
Let $X$ be an affine $G_{e}$-spherical variety with associated colored cone $(\Cc,\F)$.
We assume that $e\neq 0$.
The aim of this subsection is to provide an explicit description of the set $\Rt(X)$ of
the Demazure roots of $X$ in terms of the pair $(\Cc, \F)$, see Theorem \ref{th:final}. 
\begin{rappel}
\label{raplnd}
Let $\partial$ be the LND on $\KK[X]$ corresponding to a non-trivial normalized $\G_{a}$-action
of degree $\chi^{\ee}$, where $\ee\in M$ is nonzero. 
Equivalently, $\partial$ is nonzero, homogeneous of degree $\ee$ with respect to the $M$-grading
of $\KK[X]$, and satisfies
$[\partial,\partial_{-}] = [\partial,\partial_{+}] = 0,$
where $\partial_{\pm}$ is the LND corresponding to the $U_{\pm}$-action on $X$. 
\end{rappel} 
Applying Lemma \ref{lemma: LNDBweight} to the  Borel subgroups $B_{-}$
and $B_{+}$ of $G_{e}$, we immediately obtain the following result.
\begin{corollary}
\label{corroot}
Assume that $X$ is of type {\rm I}. Without loss of generality, we may suppose that $\KK[X]$ admits a presentation 
as in \ref{theorem:SL}. Then in the notation of Section \ref{sec:3} and  \ref{raplnd},
we have $\ee\in \Rt(\omega_{-})\cap \Rt(\omega_{+})$, and
there exists a facet $\varsigma_{\pm}\subset\omega_{\pm}^{\vee}$
such that 
$$ {\rm ker}\,\partial_{\pm}\cap {\rm ker }\,\partial = \bigoplus_{m\in\varsigma_{\pm}\cap L}\KK\phi_{m}^{\pm}\chi^{m}.$$  
\end{corollary}
\begin{proof}
Recall that two LNDs $\partial_1,\partial_2$ are called equivalent if their kernels coincide, and in such a case there holds $a_1\partial_1=a_2\partial_2$ for some $a_1,a_2\in\ker\partial_1$ and they commute. Let us prove that $\partial$ is not equivalent to neither of $\partial_+,\partial_-$. 
Indeed, assume that $\partial=\frac{f}{g}\partial_-$, $f,g\in\ker\partial_-$, without loss of generality.
Then
\[
0=[\partial_+,\partial] = \partial_+\left(\frac{f}{g}\partial_-\right) - \frac{f}{g}\partial_-\partial_+ = \partial_+\left(\frac{f}{g}\right)\partial_- + \frac{f}{g}\delta,
\]
so $\delta$ and $\partial_-$ are equivalent, a contradiction.
Since $\partial_\pm$ and $\partial$ commute, $\ker\partial_\pm$ is $\partial$-stable, the restriction of $\partial$ is a homogeneous LND. 
It is non-trivial, otherwise $\ker \partial_\pm\subsetneq\ker\partial$ and by \cite[Section~1.4, Principle~11(e)]{Fre},
\[1=\mathrm{tr.deg}_{\ker\partial_\pm} \ker\partial=\mathrm{tr.deg}_{\ker\partial_\pm} \K[X]=\mathrm{tr.deg}_{ \ker\partial}\K[X],\]  a contradiction.
Since $\ker\partial_\pm$ is a semigroup algebra, we conclude by \cite[Lemma 2.6]{Li}.
\end{proof}
The following theorem gives examples of affine spherical homogeneous spaces admitting a non-trivial
normalized $\G_{a}$-action.
\begin{theorem}
\label{theoexhom}
Let $X$ is be a $G_{e}$-spherical homogeneous space of type {\rm I}. 
Consider the representation $\KK[X] = A[\AA^{1},\D]$  as in 
Remark \ref{marker:remarkG/H}, where $\D = \D_{G_{e}\,/H}$. Then $X$ admits a normalized $\G_{a}$-action if and only if 
$X$ is reflexive and $v_{0},v_{1}$ are linearly independent.
Under this condition, all normalized $\G_{a}$-actions are horizontal,
and we have 
$$\Rt(X) = \Rti(X) = \{\ee\in L\,|\, v_{0}(\ee) = v_{1}(\ee) = 1\,\,\,{\rm or}\,\,\, v_{0}(\ee) = v_{1}(\ee) = -1\}.$$ 
\end{theorem}

\begin{proof}
Let $\partial$ be the LND corresponding to a non-trivial normalized $\G_{a}$-action on $X$. 
Denote by $\ee$ the degree of $\partial$, and by $\varpi$
the weight cone of $\ker\,\partial$. By \cite[Lemma 3.1]{Li}, $\partial$ is horizontal
with respect to the $M$-grading of $\KK[X]$.

Assume that $X$ is skew. Then the quasifan $\Lambda(\D)$ (see \ref{marker:defpol}) has exactly two 
maximal elements. By Theorem~\ref{theorem:hor}, they are exactly $\omega_{-}^{\vee}$ and $\omega_{+}^{\vee}$, hence $\varpi$
coincides with one of them. But this contradicts Corollary~\ref{corroot}.
Thus, we may suppose that $X$ is reflexive and $\Lambda(\D)$ has at least three maximal cones.
The latter condition is equivalent to $v_{0}, v_{1}$ being linearly independent.
In fact, in this case $\Lambda(\D)$ has four maximal cones,  and $\omega_-$, $\omega_+$, $\varpi$ are among them.

By Theorem \ref{theorem:hor}, the LND $\partial$ can be given by a triple $(\lambda, \dcol, \ee)$, where
we let $\lambda =1$ for simplicity and
where the colored polyhedral divisor is given by 
$$\dcol  = (\D, \{\mu_{z}\,|\, z\in\AA^{1}\}, z_{0}).$$

Since $\omega_-$ and $\omega_+$ correspond to diagrams
$$
\dots \lin \cou{v_0}{0} \lin \cou{0}{1} \lin \dots\qquad \qquad\dots \lin \cou{0}{0} \lin \cou{v_1}{1} \lin\ldots
$$
the two possible cases remaining for $\varpi$ are
$$
\dots \lin \cou{0}{0} \lin \cou{0}{1} \lin \dots\qquad \qquad\dots \lin \cou{v_0}{0} \lin \cou{v_1}{1} \lin\ldots
$$
 where the marked point $z_0$ is not yet known.

\emph{Case {\rm 1}.} 
The coloring $\dcol$ is represented by one of diagrams
$$
\dots \lin \cou{0}{z_0=0} \lin \cou{0}{1} \lin \dots
\dots \lin \cou{0}{0} \lin \cou{0}{z_0=1} \lin \dots
\dots \lin  \cou{0}{0} \lin \cou{0}{1} \lin \cou{0}{z_{0}} \lin \dots 
$$
The LND $\partial$ is given by the relation
$$
\partial(Q\chi^{m})= Q'\chi^{m+\theta} \mbox{ for all } Q\in\K(t), m\in M.
$$
Let us study the conditions $[\partial,\partial_{+}] = [\partial,\partial_{-}] = 0$.
Recall that
\begin{align*}
\partial_-(Q\chi^{m})=&(v_0(m)Q+tQ')\chi^{m-e},  \mbox{ for all } Q\in\K(t), m\in M,\\
\partial_+(Q\chi^{m})=&(v_1(m)Q+(t-1)Q')\chi^{m+e} \mbox{ for all } Q\in\K(t), m\in M.
\end{align*}
Then
\begin{align*}
\partial\circ \partial_{-} (Q\chi^{m})=&
(v_0(m)Q'+Q'+tQ'')\chi^{m+\theta-e},\\
\partial_{-}\circ \partial (Q\chi^{m})=&
(v_0(m+\theta)Q'+tQ'')\chi^{m+\theta-e}.\\
\partial\circ \partial_{+} (Q\chi^{m})=&
(v_1(m)Q'+Q'+(t-1)Q'')\chi^{m+\theta+e},\\
\partial_{+}\circ \partial (Q\chi^{m})=&
(v_1(m+\theta)Q'+(t-1)Q'')\chi^{m+\theta+e}.
\end{align*}
The conditions $[\partial,\partial_{-}]=0$ and $[\partial,\partial_{+}]=0$ are therefore equivalent
$v_0(\theta)=1$ and $v_1(\theta)=1$ respectively.

\emph{Case {\rm 2.}} The set of AL-colors $\{\mu_{z}\,|\, z\in\AA^{1}\}$ is $\{v_{0},v_{1}\}$.
Since $\varsigma_{\pm}=\varpi\cap\omega_{\pm}^{\vee}$ is a facet that separates full-dimensional cones $\varpi$ and $\omega_{\pm}^{\vee}$ in $\sigma^{\vee}$,
we have $\varsigma_{\pm}\cap\relint(\sigma^{\vee})\neq\emptyset.$ This implies
that 
$$\ker\,\partial = \bigoplus_{m\in\varpi\cap L}\KK t^{-v_{0}(m)}(t-1)^{-v_{1}(m)},$$
and therefore $z_{0}\in \{0,1\}$ by Remark \ref{rem:general}. 
By applying the formula from Remark~\ref{rem:general}, one can see that in both cases $z_0=0$ and $z_0=1$ the derivation $\partial$
is the same. 
namely for all  $Q\in\K(t), m\in M$
$$
\partial(Q\chi^{m})= \left(\left(\frac{v_0(m)}{t}+\frac{v_1(m)}{t-1}\right)Q+Q'\right)\gamma\chi^{m+\theta} \mbox{ where } \gamma=t^{-v_0(\theta)}(t-1)^{-v_1(\theta)}.
$$
Let us study the conditions $[\partial,\partial_{-}] = [\partial,\partial_{+}] = 0$.
Using $\gamma'=-v_0(\theta)\frac{\gamma}{t}-v_1(\theta)\frac{\gamma}{t-1},$ we have
\begin{align*}
\partial\circ \partial_{-} (t)=&\partial(t\chi^{-e})=
\frac{t}{t-1}\gamma\chi^{\theta-e},\\
\partial_{-}\circ \partial (t)=&\partial_-(\gamma\chi^{\theta})=
-v_1(\theta)\frac{t}{t-1}\gamma\chi^{\theta-e},\\
\partial\circ \partial_{+} (t)=&\partial((t-1)\chi^{e})= 
\frac{t-1}{t}\gamma\chi^{\theta+e},\\
\partial_{+}\circ \partial (t)=&\partial_+(\gamma\chi^{\theta})=
-v_0(\theta)\frac{t-1}{t}\gamma\chi^{\theta+e}.
\end{align*}
The conditions $\partial\circ\partial_{+}(t)=\partial_{+}\circ\partial(t)$ and $\partial\circ\partial_{-}(t)=\partial_{-}\circ\partial(t)$ are therefore equivalent
$v_0(\theta)=-1$ and $v_1(\theta)=-1$ respectively.
Thus, $\gamma=t(t-1)$ and
$$
\partial(Q\chi^{m})= ((v_0(m)(t-1)+v_1(m)t)Q+t(t-1)Q')\chi^{m+\theta}.
$$
It is easy to check that $\partial\circ\partial_{-}(\chi^{m})=\partial_{-}\circ\partial(\chi^{m})$ and $\partial\circ\partial_{+}(\chi^{m})=\partial_{+}\circ\partial(\chi^{m})$ for any $m\in M$. Therefore, $\partial$ commutes with $\partial_-$ and $\partial_+$.

Finally, from cases 1 and 2 we obtain the condition $v_{0}(\ee) = v_{1}(\ee) = \pm 1$,
which implies that $\ee$ is the Demazure root of the 
 domain of $\D_{\bullet}$. Now the proof of the theorem is completed.
\end{proof}

Let us give an example of a non-trivial interior normalized $\G_{a}$-action.
\begin{example}
\label{exinterior}
Assume that $M = N = \mathbb{Z}^{2}$ and  $M_{\mathbb{Q}} = N_{\mathbb{Q}} = \mathbb{Q}^{2}$.
The pairing between $M$ and $N$ is given by the usual scalar product. Consider the polyhedral
divisor 
$$\D = {\rm Conv}(0,v_{0})\cdot \{0\} + {\rm Conv}(0,v_{1})\cdot \{1\}$$
over the affine line $\AA^{1}$, where $v_{0} = (-1,0)$ and $v_{1} = (0,-1)$.
An easy computation shows that 
$A = A[\AA^{1},\D] = \KK[u_{1},u_{2},u_{3},u_{4}]$, where $u_{1} = \chi^{(-1,0)}$, $u_{2} = \chi^{(0,-1)}$,
$u_{3} = (t-1)\chi^{(0,1)}$, and $u_{4} = t\chi^{(1,0)}$ satisfy the irreducible relation $u_{1}u_{4}-u_{2}u_{3}-1$.
Hence $X = {\rm Spec}\, A$ is identified with the hypersurface $\VV(x_{1}x_{4} - x_{2}x_{3} - 1)\subset \AA^{4}$.  
The action of the torus $\TT = (\KK^{\star})^{2}$ on the coordinates is given by the formula
$$(t_{1},t_{2})\cdot (x_{1},x_{2},x_{3},x_{4}) = (t_{1}\cdot x_{1}, t_{2}\cdot x_{2}, 
t_{2}^{-1}\cdot x_{3}, t_{1}^{-1}\cdot x_{4}),$$
where $(t_{1},t_{2})\in \TT$. Consider the vector $e = (-1,1)$. Then we have
$\langle v_{0}, e\rangle  = 1$ and $\langle v_{1}, e \rangle  = -1$. The corresponding LNDs are
given by 
$$\partial_{-}  = u_{2}\frac{\partial}{\partial u_{1}} + u_{4}\frac{\partial}{\partial u_{3}}\,\,\,{\rm and}
\,\,\,\partial_{+}  = u_{1}\frac{\partial}{\partial u_{2}} + u_{3}\frac{\partial}{\partial u_{4}}.$$ 
Thus, the action of the algebraic group $G_{e} = \SL_{2}\rtimes_{e}\TT$ is defined by the 
formula
$$g\cdot (x_{1},x_2,x_3, x_{4})
 = (dt_{1}x_{1} + ct_{2}x_{2}, bt_{1}x_{1} + at_{2}x_{2}, dt_{2}^{-1}x_{3} + ct_{1}^{-1}x_{4}, bt^{-1}_{2}x_{3} 
+ at_{1}^{-1}x_{4}),$$
where $g= \left(\left(\begin{smallmatrix}
              a   &   b \\
              c   &   d\end{smallmatrix}\right),(t_1,t_2)\right)\in G_{e}$. In particular,
we have 
$$g\cdot (1,0,0,1) = (dt_{1}, bt_{1}, ct_{1}^{-1}, at_{1}^{-1}).$$
The isotropy subgroup of $(1,0,0,1)$ is equal to 
$$H  = \left\{\left.\left(\begin{pmatrix}
              t_{1}   &   0 \\
              0   &   t_{1}^{-1}\end{pmatrix}, (t_{1},t_{2})\right)\,\,\right|\,\, t_{1},t_{2}\in\KK^{\star}\right\},$$
and we have a $G_{e}$-isomorphism $X\simeq G_{e}\,/H$.
Let us describe the non-trivial normalized $\G_{a}$-actions on $X$. By Theorem \ref{theoexhom},  we observe that 
$\Rti(X) = \{(1,1), (-1,-1)\}$.
The associated LNDs to the vectors $(-1,-1)$, $(1,1)$ are given respectively by
$$ u_{3}\frac{\partial}{\partial u_{1}} + u_{4}\frac{\partial}{\partial u_{2}}\,\,\,{\rm and}\,\,\,
u_{1}\frac{\partial}{\partial u_{3}} + u_{2}\frac{\partial}{\partial u_{4}},$$  
and the associated $\G_{a}$-actions by
$$\lambda\star (x_{1},x_{2},x_{3},x_{4}) = (x_{1} + \lambda x_{3}, x_{2} + \lambda x_{4}, x_{3}, x_{4})$$
and  $$\lambda\star (x_{1},x_{2},x_{3},x_{4}) = (x_{1}, x_{2}, x_{3} + \lambda x_{1}, x_{4} + \lambda x_{2})$$ 
for any $\lambda\in\G_{a}$. 

Let us describe the $\G_{a}\rtimes_{\ee}G_{e}$-homogeneous space
$X = X_{\rho}$ associated with the root $\ee\in\Rti(X)$ with the distinguished ray $\rho$. If $\ee = (-1,-1)$, then we have 
$$(\lambda, g)\cdot (1,0,0,1) = \lambda\star(g\cdot (1,0,0,1)) = 
(dt_{1}+\lambda c t_{1}^{-1}, bt_{1} + \lambda a t_{1}^{-1}, ct_{1}^{-1}, a t_{1}^{-1}).$$
The isotropy subgroup of $(1,0,0,1)$ for the $\G_{a}\rtimes_{\ee}G_e$-action is
$$H_{1} = \left\{\left.\left(-bt_{1}, \begin{pmatrix}
              t_{1}   &   b \\
              0   &   t_{1}^{-1}\end{pmatrix}, (t_{1}, t_{2})\right)\,\,\right|\,\, b\in\KK,\,(t_{1},t_{2})\in
(\KK^{\star})^{2}\right\}.$$
We have a $\G_{a}\rtimes_{\ee} G_{e}$-isomorphism $X\simeq (\G_{a}\rtimes_{\ee}G_{e})/H_{1}$.
Similarly, for $\ee = (1,1)$ the isotropy subgroup of $(1,0,0,1)$ is
$$H_{2} = \left\{\left.\left(-ct_{1}^{-1}, \begin{pmatrix}
              t_{1}   &   0 \\
              c   &   t_{1}^{-1}\end{pmatrix}, (t_{1}, t_{2})\right)\,\,\right|\,\, c\in\KK,\,(t_{1},t_{2})\in
(\KK^{\star})^{2}\right\},$$
and $X\simeq (\G_{a}\rtimes_{\ee}G_{e})/H_{2}$.
\end{example}
The next two lemmas will be used in the proof of Proposition \ref{prop:final}. 
\begin{lemme}
\label{lemmaexterior}
Let $X$ be an affine $G_{e}$-spherical variety of type $\rm I$. Assume that the torus $\TT$ acts faithfully on $X$.
Then every exterior normalized $\G_{a}$-action on $X$ is vertical for the induced $\TT$-action. 
\end{lemme}
\begin{proof}
According to Theorem \ref{theoroot}, it is enough to consider the $\G_{a}\rtimes_{{\ee}}G_{e}$-homogeneous 
variety $X_{\rho}$, where  $\rho$ is the distinguished
ray of $\ee\in\Rte(X)$. Note that a priori $X_\rho$ is not necessarily affine. By Theorem~\ref{theoroot}(ii) and Lemma~\ref{lemma:valhorvert}, we have $\rho\in \linea(\Vc) = e^{\perp}$. 
This implies that
$X_{\rho}$ is affine and described by a $\sigma$-polyhedral divisor over $\AA^{1}$,
where $\sigma  = \QQ_{\geq 0}\cdot\rho$. By \cite[Lemma 3.1]{Li}, the equality $\rho(\ee) = -1$ implies that 
the $\G_{a}$-action on $X_{\rho}$ is vertical. 
\end{proof}

\begin{lemme}
\label{lemmabracket}
Let $\phi\in\KK(t)^{\star}$, $s\in\mathbb{Z}$, $d\in\mathbb{Z}\setminus\{0\}$, $v\in N_{\QQ}$,
$\rho\in N_{\QQ}\setminus\{0\}$, and $e,\ee\in M$. 
Consider the derivations $\partial, \partial_{\bullet}$ on
$\KK(t)[M]$ defined by the relations 
\begin{align*}
\partial_{\bullet}(t^{r}\chi^{m}) = d\cdot (v(m)+r)\chi^{m+e}t^{r+s},
 \\
\partial(t^{r}\chi^{m}) = \rho(m)\phi \,t^{r}\chi^{m+\ee}\,\,\,{\rm for}\,\,\,{\rm all}\,(m,r)\in M\oplus\mathbb{Z}.
\end{align*}
Then $[\partial,\partial_{\bullet}] = 0$
if and only if $\rho\in e^{\perp}$, $v(\ee)\in\mathbb{Z}$, 
and $\phi = \lambda t^{-v(\ee)}$ for some $\lambda\in \KK^{\star}$. 
\end{lemme}
\begin{proof}
Let us compute the Lie bracket $[\partial_{\bullet},\partial](t^{r}\chi^{m})$.
For any $(m,r)\in M\oplus\mathbb{Z}$, we have
\begin{align*}
\partial_{\bullet}\circ \partial(t^{r}\chi^{m}) =& 
d\cdot\rho(m)\left[\frac{d\phi}{dt}t+(v(m+\ee)+r)\phi\right]t^{r+s}\chi^{m+e+\ee},\,\,\,{\rm and}
\\
\partial\circ\partial_{\bullet}(t^{r}\chi^{m}) =& d\cdot\rho(m+e)(v(m)+r)\phi t^{r+s}\chi^{m+e+\ee}.
\end{align*}

Thus, the condition $[\partial,\partial_{\bullet}](t^{r}\chi^{m}) = 0$ is equivalent to
$$
\rho(m)\left(v(\ee)\phi+t\frac{d\phi}{dt}\right)=\rho(e)(v(m)+r)\phi
$$
for all $r\in\ZZ$, $m\in M$. The left-hand side does not depend on $r$, hence $[\partial,\partial_{\bullet}] = 0$
implies $\rho(e)r\phi=0$ for any $r\in\ZZ$ and is equivalent to
$$\begin{cases} \rho(e)=0,\\ 
v(\ee)\phi+t\frac{d\phi}{dt}=0.\end{cases}$$
The latter equation is equivalent to $\phi= \lambda t^{-v(\ee)}$ for some $\lambda\in \KK^{\star}$. 
%
%
\end{proof}
\begin{proposition}
\label{prop:final}
Let $X$ be an affine $G_{e}$-spherical variety. Assume that $e\neq 0$ and that the $\TT$-action on $X$ is faithful of complexity one. Let $(\Cc, \F)$ be the associated colored cone of $X$ with respect to the 
$B_{\pm}$-action. Let $\Gamma$ be the polyhedral cone generated by the subset $\Cc$ and by the images of all colors of $X$.
Then the following  hold.
\begin{itemize}
 \item [(i)] If $X$ is of type {\rm I}, then  $X$ is isomorphic to an embedding of $G_{e}\,/H$,
where $G_{e}\,/H$ is an in \ref{marker:remarkG/H}. In the reflexive case,
$$\Rte(X) = \{\ee\in\Rt(\Gamma)\cap L\,|\, \rho_{\ee}\in\linea(\Vc),\,\, v_{0}(\ee)=v_{1}(\ee) = 0\},$$
 and  in the skew case
$$\Rte(X) =\{\ee\in\Rt(\Gamma)\cap L\,|\, \rho_{\ee}\in\linea(\Vc),\,\, 
v_{1}(\ee)= 0\}.$$
 Furthermore, if $X$ is skew or $v_{0},v_{1}$ are linearly dependent, then 
$\Rti(X) = \emptyset$; otherwise, we have 
$$\Rti(X) = \{\ee\in \Rt(\Gamma)\cap L\,|\,v_{0}(\ee) = v_{1}(\ee) = -1\,\,\, {\rm or}\,\,\, v_{0}(\ee) = v_{1}(\ee) = 1\}.$$
\item [(ii)] 
If $X$ is not of type {\rm I}, then $X$ is a $\tilde G_{\tilde e}$-spherical variety of type \emph{II (}see Lemma~\ref{lemmenotoftypeI}{\rm)}  isomorphic to an embedding of  $X_{\Sigma}$ as in Lemma~\ref{lemmatoricdiv}.
We have $\Rti(X) = \emptyset$ and, in
the notation of \ref{proptoric},
$$\Rte(X) =\{\ee\in\Rt(\Gamma)\cap L_{\pm}\,|\,\ee\in\rho_{-}^{\perp}\cap \rho_{+}^{\perp}\,\,\,{\rm and}\,\,\, \rho_{\ee}\in e^{\perp} \}.$$
\end{itemize} 
\end{proposition}
\begin{proof}
(i) Note that the description of $\Rti(X)$ is a direct consequence of Theorem~\ref{theoexhom}. 
Let us describe the exterior Demazure roots of $X$.
Let $\ee\in \Rte(X)$ with the distinguished root $\rho=\rho_\ee$. Then by Lemma \ref{lemmaexterior}, the corresponding $\G_{a}$-action on $X_{\rho}$
is obtained from a vertical LND $\partial$ on $\KK[X_{\rho}]$ of degree $\ee$ such that the algebra $\KK[X]$ is $\partial$-stable, see \ref{marker: Demazure}.
By Lemma \ref{lemmabracket}, the condition $[\partial, \partial_{-}] = 0$ implies that $v_{0}(\ee)\in\mathbb{Z}$
and  $\partial$ is given by the formula
$$\partial(Q\chi^{m}) = \rho(m)t^{-v_{0}(\ee)}Q\chi^{m+\ee},$$ 
where $Q\in\KK(t)$ and $m\in M$. The other implication $\rho  \in e^{\perp}$
 is equivalent to $\rho\in\linea(\Vc)$.

\emph{Reflexive case.} 
Using formula for $\partial_+$ in \ref{rap:der-Q} and the equality $\rho(e)=0$, we have
\begin{align*}
[\partial_+,\partial](Q\chi^{m})=\rho(m)\left(v_1(\ee)-v_0(\ee)\frac{t-1}{t}\right)t^{-v_0(\ee)}Q\chi^{m+e+\ee}\mbox{ for all } Q\in\K(t), m\in M.
\end{align*}
Therefore, $\partial$ and $\partial_+$ commute if and only if $v_0(\ee)=v_1(\ee)=0$.
We conclude by remarking
that $\partial(\KK[X])\subset \KK[X]$ gives the condition $\ee\in \Rt(\Gamma)$.


\emph{Skew case.} 
Similarly,
\begin{align*}
[\partial_+,\partial](Q\chi^{m})=2\rho(m)v_1(\ee)t^{-v_0(\ee)}Q\chi^{m+e+\ee}\mbox{ for all } Q\in\K(t), m\in M.
\end{align*}
Thus, $\partial$ and $\partial_+$ commute if and only if $v_1(\ee)=0$.
Moreover, let $\D'$ be the polyhedral divisor over $\AA^{1}$ associated with $X_{\rho}$;
here we identify $\KK[X_{\rho}]$ with a subalgebra of $\KK[G_{e}\,/H]$. 
Note that the condition on the vertical derivation $\partial=\partial_{\ee,t^{-v_0(\ee)}}$ imposed by \ref{marker: Demazure}, namely
$$t^{-v_{0}(\ee)}\in
\Phi_{\ee}^{\star} = H^{0}(\AA^{1},\mathcal{O}_{\AA^{1}}(\lfloor\D'(\ee)\rfloor))\setminus\{0\},$$ 
is always fulfilled. 
Again, we have the condition $\ee\in\Rt(\Gamma)$.

(ii) $\KK[X]$ is the semigroup algebra $\KK[\sigma^{\vee}\cap M]$, where $\sigma^{\vee}\subset M_{\QQ}$
is the weight cone for the $\TT$-action. Let $\ee\in\Rt(X)$ with the distinguished root $\rho=\rho_\ee$. Then $\ee\in\Rt(\sigma)$ and it corresponds to
the homogeneous LND $\partial$ given by
 $$\partial(\chi^{m}) = \rho(m)\chi^{m+\ee}.$$
The conditions $[\partial,\partial_{-}] = [\partial,\partial_{+}] = 0$ are equivalent to
(see the notation of \ref{marker: semisimpleroot})
$$
\begin{cases}
\rho_{-}(\ee)\rho(m) = - \rho(e)\rho_{-}(m)\mbox{ for any }m\in\sigma^{\vee}\cap M, \\
\rho_{+}(\ee)\rho(m) = \rho(e)\rho_{+}(m)\mbox{ for any }m\in\sigma^{\vee}\cap M.
\end{cases}$$
In other words, $\rho_{-}(\ee)\rho = - \rho(e)\rho_{-}$ and $\rho_{+}(\ee)\rho = \rho(e)\rho_{+}$. Moreover, these two vectors are collinear. Since $\rho_-$ and $\rho_+$ are not, $\rho(e)=0$.
Moreover,  $\rho\neq0$ implies $\rho_{-}(\ee)=\rho_{+}(\ee)=0$.
But $\rho(\ee)=-1$ by definition, hence
$\rho\notin\{\rho_-,\rho_+\}$. So, Lemma~\ref{lemmatoricdiv}(iii) implies that $\rho$ is exterior and $\Rti(X) = \emptyset$. 
Finally, under the conditions $\rho(e)=\rho_{-}(\ee)=\rho_{+}(\ee)=0$ the induced $\G_a$-action is indeed normalized by $G_e$.
\end{proof}
The following result gives a description of Demazure roots 
of affine $G_{e}$-spherical varieties of types I and II. 

\begin{theorem}\label{th:final}
Let $X$ be an affine embedding of a $G_e$-homogeneous spherical space $G_e/H$
corresponding to the colored cone $(\Cc,\F)$. Assume that the $\TT$-action on $X$ is faithful of complexity one. Denote by $\F_0$ the subset of colors of $G_e/H$ and $\Gamma=\Cone(\Cc\cup\varrho(\F_0))$.
Then the set of exterior Demazure roots is 
\[
\Rte(X) = \{\ee\in\Rt(\Gamma)\cap L\,|\, \rho_{\ee}\in\linea(\Vc),\,\,\ee\in\varrho(\F_0)^{\perp}\},
\]
whereas the set of interior Demazure roots is non-empty if and only if $H$ is a torus that does not contain a maximal torus $T$ of $\SL_2$. In such a case $\F_0$ consists of two different colors, say $D_1,D_2$ and
\[
\Rti(X) = \{\ee\in\Rt(\Gamma)\cap L\,|\, \rho_{\ee}\in\linea(\Vc),\,\,\{\varrho(D_1)(\ee),\varrho(D_2)(\ee)\}=\{-1,1\}\}.
\]
\end{theorem}
\begin{proof}
This is the immediate consequence of Proposition~\ref{prop:final} and Theorem~\ref{th:sph-subgroups-I}.
\end{proof}
\begin{remark}\label{rem:red}
Let us consider an arbitrary $G_e$-spherical variety $X$.
 The $\TT$-action on $X$ is of complexity at most two, since $\codim_B\TT=2.$
If the $\TT$-action on $X$ is of complexity $2$, then we take $\TT'=\TT\times T$, i.e. we add a copy of the torus $T\subset\SL_2$ to $\TT$ preserving its action on $X$.
Then $e'=(e,-1)$, $G_{e'}=\SL_2\rtimes \TT'$, and $X$ is the $G_{e'}$-spherical variety. Since $\codim_B T\times\TT=1$  in $G_e$, the $\TT'$-action on $X$ is of complexity one. There is a bijection between Demazure roots of $X$ corresponding to  $G_{e}$ and $G_{e'}$, sending a root $\chi$ to $(\chi,0)$. Indeed, since  normalized $\G_a$-actions commute with the $\SL_2$-action, they commute with $T$ as well.

Assume now that the $\TT$-action on $X$ is not faithful. Denote $G_e^\prime=\SL_2\times\TT'$, where $\TT'$ is the quotient of $\TT$ by the kernel of non-effectivity. Then $X$ is a $G_e^\prime$-spherical variety endowed with the faithful $\TT'$-action and Demazure roots with respect to $G_e$ and $G_e^\prime$ coincide. The complexity of $\TT$-action is not changed under this reduction.
\end{remark}
\begin{corollaire}\label{cor:final}
Let $e\in M$ be a nonzero element. Then
the answer to Question \ref{question} is affirmative for the reductive group $G_{e}$.  
\end{corollaire}
\begin{proof}
This follows from Theorem~\ref{th:final} and Remark~\ref{rem:red}.
\end{proof}

\end{document}